\documentclass[a4paper, 10pt]{amsart}
\usepackage{amsmath,  color}
\usepackage{amssymb, amscd, hyperref}

\theoremstyle{plain}
\newtheorem{theorem}{\bf Theorem}[section]
\newtheorem{proposition}[theorem]{\bf Proposition}
\newtheorem{lemma}[theorem]{\bf Lemma}
\newtheorem{corollary}[theorem]{\bf Corollary}
\newtheorem{conjecture}[theorem]{\bf Conjecture}

\theoremstyle{definition}
\newtheorem{example}[theorem]{\bf Example}

\newtheorem{definition}[theorem]{\bf Definition}

\newcommand{\N}{\mathbb N}

\newcommand{\C}{\text{\rm C}}

\newcommand{\DP}{\negthinspace :\negthinspace}
\newcommand{\und}{\;\mbox{ and }\;}

\newcommand{\red}{{\text{\rm red}}}

\newcommand{\mon}{\text{\rm mon}}
\newcommand{\adj}{\text{\rm adj}}
\newcommand{\eq}{\text{\rm eq}}
\newcommand{\G}{\text{\rm G}}
\newcommand{\BF}{\text{\rm BF}}
\newcommand{\FF}{\text{\rm FF}}

\newcommand{\f}{\mathfrak}

 \DeclareMathOperator{\ord}{ord}

\DeclareMathOperator{\spec}{spec} \DeclareMathOperator{\supp}{supp}
 
\DeclareMathOperator{\Ker}{Ker} 
\DeclareMathOperator{\End}{End}

\numberwithin{equation}{section}

\begin{document}

\title{On conductor submonoids of factorial monoids}

\author{Alfred Geroldinger and  Weihao Yan and Qinghai Zhong}

\address{University of Graz, NAWI Graz\\
Department of Mathematics and Scientific Computing\\
Heinrichstra{\ss}e 36\\
8010 Graz, Austria (A.~Geroldinger and Q.~Zhong)}
\email{alfred.geroldinger@uni-graz.at, qinghai.zhong@uni-graz.at}
\urladdr{https://imsc.uni-graz.at/geroldinger,  https://imsc.uni-graz.at/zhong/}

\address{School of Mathematical Sciences, Hebei Normal University, Shijiazhuang, Hebei Province, 050024 China}
\email{weihao.yan.hebnu@outlook.com}



\thanks{The first and third authors were supported  by the Austrian Science Fund FWF, Project P36852-N}

\begin{abstract}
We study algebraic and arithmetic properties of submonoids (resp. subrings) of factorial monoids (resp. factorial domains) whose non-invertible elements all lie in the conductor. This continues earlier work of Baeth, Cisto, et al.. On our way we answer several conjectures, formulated in their papers in the affirmative (\cite[Conjecture 4.16]{Ba23a} and  \cite[Conjectures 2.3 and 2.10, and Section 9]{Ci-GS-LL25a}).
\end{abstract}

\maketitle

\smallskip
\section{Introduction} \label{1}
\smallskip

Two classes of submonoids of free abelian monoids received some attention in recent literature.   To begin with, a  submonoid $H$ of a finitely generated free abelian monoid $F$ is called cofinite if $F \setminus H$ is finite. Numerical monoids, generalized numerical monoids, and cofinite affine monoids are classic examples. Part of the motivation stems from combinatorial commutative algebra where these monoids are studied together with their associated monoid algebras (see \cite{Fa-Pe-Ut16a,Ka15a, Ci-Fa-Na25a,Ro-TR-VT25,GG-TR-VT25} for recent contributions). Secondly, a  submonoid  $H$ of a finitely generated free abelian monoid $F$ is called an ideal extension of $F$ if $H = \mathfrak a \cup \{\boldsymbol 0\}$ for some $s$-ideal $\mathfrak a \subset F$ (if this holds, then the $s$-ideal $\mathfrak a$ equals the conductor $(H \DP F)$).

\smallskip
N.~Baeth  initiated the investigation of submonoids of free abelian monoids that are both ideal extensions and cofinite. Indeed, in \cite{Ba23a} he studied complement-finite ideals. 
The work of Baeth was recently continued by C.~Cisto et al. \cite{Ci-GS-LL25a}   who introduced the concept of ideal extensions of $(\N_0^{(I)}, +)$. The  observation that an ideal extension monoid $H \subset \N_0^{(I)}$  is $\le$-convex (this means that for each two elements $\boldsymbol a, \boldsymbol b \in H$, the whole interval $\{ \boldsymbol c \in \N_0^{(I)} \colon \boldsymbol a \le \boldsymbol c \le \boldsymbol b \} $ lies in $H$) was their motivation to introduce the concept of gap absorbing monoids (see Definition \ref{3.3}). 

\smallskip
In the present paper,  we introduce submonoids of (not necessarily finitely generated) factorial monoids whose non-invertible elements all lie in the conductor to the factorial monoid, and call them conductor submonoids (see Definition \ref{3.3}). This approach generalizes the concept of ideal extension monoids and in Theorem \ref{3.8} we show that a submonoid of a factorial monoid is a conductor submonoid if and only if it is a gap absorbing monoid in the sense of \cite[Section 2]{Ci-GS-LL25a} (this settles Conjecture 2.10 in \cite{Ci-GS-LL25a}). In addition to the classes of monoids comprised by the former concepts, our more general approach includes well-studied classes of  local domains and various monoids of ideals (Examples \ref{3.9}, \ref{3.11},  and \ref{4.4}, Proposition \ref{4.3}, and Corollary \ref{5.2}). 

\smallskip
In Section \ref{3}, we study algebraic properties of conductor submonoids of  factorial monoids (Theorem \ref{3.5}) and we characterize them by the  gap absorbing property (Theorem \ref{3.8}). In Section \ref{4}, we consider conductor submonoids lying between the monoid of principal ideals and the monoid of divisorial ideals of a Krull monoid (Proposition \ref{4.3} and Theorem \ref{4.7}). Section \ref{5} is devoted to the arithmetic of conductor submonoids. First we show that conductor submonoids of factorial monoids have catenary degree at most three (Theorem \ref{5.1}), which gives a positive answer to \cite[Conjecture 4.16]{Ba23a} and to Question 3 in \cite[Section 9]{Ci-GS-LL25a}. Finally, Theorem \ref{5.5} provides detailed information on the arithmetic of conductor submonoids associated with Krull monoids.

\smallskip
\section{Preliminaries} \label{2}
\smallskip

By a {\it monoid}, we mean a commutative and cancellative semigroup with identity. Let $H$ be a monoid. We denote by $H^{\times}$ the group of invertible elements, by $\mathsf q (H)$ the quotient group of $H$, by $\mathcal A (H)$ the set of atoms of $H$, and by $H_{\red} = H/H^{\times} = \{ a H^{\times} \colon a \in H\}$ the associated reduced monoid of $H$. For subsets $A, B \subset H$, let
\[
A \, B = \{ a b \colon a \in A, b \in B \} \subset H
\]
be the {\it product set} of $A$ and $B$. Furthermore, let
\begin{itemize}
\item $\widehat H = \{ x \in \mathsf q (H) \colon \text{there is some $c \in H$ such that $cx^n \in H$ for all $n \in \N$} \} \subset \mathsf q (H)$ denote the {\it complete integral closure} of $H$, and let
    
\item $H ' = \{ x \in \mathsf q (H) \colon \text{there is some $N \in \N$ such that $x^n \in H$ for all $n \ge N$} \} \subset \mathsf q (H)$ denote the {\it seminormalization} of $H$.    
\end{itemize}
Then $H \subset H' \subset \widehat H \subset \mathsf q (H)$ and $H$ is called {\it seminormal} (resp. {\it completely integrally closed}) if $H=H'$ (resp. $H = \widehat H$).
For a set $P$, we denote by $\mathcal F (P)$ the free abelian monoid with basis $P$. Let $F$ be a factorial monoid. Then $F \cong F^{\times} \times \mathcal F (P)$, where $P$ is a set of representatives of prime elements of $F$. If $p \in P$, then $\mathsf v_p \colon F \to \N_0$ denotes the $p$-adic valuation. For an element $a = \varepsilon p_1^{k_1} \cdot \ldots \cdot p_s^{k_s} \in F$, where $p_1, \ldots, p_s$ are non-associated primes and $k_1, \ldots, k_s \in \N$, then
\[
|a|_F = k_1+ \ldots + k_s = \sum_{p \in P} \mathsf v_p (a) \in \N_0 
\]
denotes the {\it length} of $a$.

\smallskip
\noindent
{\bf Ideal Theory of Monoids.}  For subsets $X, Y \subset \mathsf q (H)$, we set $(X \colon Y) = \{ x \in \mathsf q (H) \colon x Y \subset X\}$,  $X^{-1} = (H \colon X)$, and $X_v = (X^{-1})^{-1}$.
A subset $\mathfrak a \subset H$ is called 
\begin{itemize}
\item an {\it $s$-ideal} of $H$ if $\mathfrak a H = \mathfrak a$,

\item a {\it $v$-ideal} (or a {\it divisorial ideal}) of $H$ if $\mathfrak a_v = \mathfrak a$.
\end{itemize}
Let $s$-$\spec (H)$ (resp. $v$-$\spec (H)$) denote the set of prime $s$-ideals (resp. prime $v$-ideals) of $H$, and let $\mathfrak X (H) \subset $ $s$-$\spec (H)$ be the set of all minimal nonempty prime $s$-ideals of $H$. The monoid $H$ is called a
\begin{itemize}
\item a {\it Mori monoid} if it satisfies the ACC (ascending chain condition) on divisorial ideals,

\item a {\it \G-monoid} if 
      \[
       \bigcap_{\substack {\mathfrak p \in s\text{\rm -spec}(H)\\ \mathfrak p \ne \emptyset}} \mathfrak p \ne \emptyset \,,
       \]

\item a {\it Krull monoid} if it is a completely integrally closed Mori monoid, and

\item {\it $v$-local} if $|v$-$\max (H)|= 1$.
\end{itemize}
A Mori monoid is $v$-local if and only if $H \setminus H^{\times}$ is a $v$-ideal (\cite[Theorem 2.2.5]{Ge-HK06a}). If $H_{\red}$ is finitely generated, then $H$ is a Mori G-monoid (\cite[Corollary 4.10]{G-HK-H-K03}).

\smallskip
\noindent
{\bf Arithmetic of Monoids.} The free abelian monoid $\mathsf Z (H) = \mathcal F ( \mathcal A (H_{\red})$ denotes the factorization monoid of $H$ and $\pi \colon \mathsf Z (H) \to H_{\red}$, defined by $\pi (u) = u$ for all $u \in \mathcal A (H_{\red})$,  is the factorization homomorphism. For an element $a \in H$, we call
\begin{itemize}
\item $\mathsf Z_H (a) = \mathsf Z (a) = \pi^{-1} (aH^{\times}) \subset \mathsf Z (H)$ the {\it set of factorizations} of $a$, and
    
\item $\mathsf L_H (a) = \mathsf L (a) = \{ |z| \colon z \in \mathsf Z (a) \} \subset \N_0$ the {\it length set} of $a$.
\end{itemize}    
The monoid $H$ is called
\begin{itemize}
\item {\it half-factorial} if $|\mathsf L (a)|=1$ for all $a \in H$, 
\item {\it length-factorial} if  $1 \le |\mathsf L (a)|=|\mathsf Z (a)|$ for all $a\in H$, 
\item a {\it \BF-monoid} if $\mathsf L (a)$ is finite and nonempty for all $a \in H$, and 
\item an {\it \FF-monoid} if $\mathsf Z (a)$ is finite and nonempty for all $a \in H$.
\end{itemize}
Every Krull monoid is an FF-monoid and every Mori monoid is a BF-monoid.
Let $z,\,z'\in\mathsf Z (H)$. Then we can write
\[
z = u_1\cdot\ldots\cdot u_{\ell} v_1\cdot\ldots\cdot v_m\quad
\text{and}\quad z' = u_1\cdot\ldots\cdot u_{\ell} w_1\cdot\ldots
\cdot w_n\,,
\]
where $\ell,\,m,\,n\in\N_0$ and $u_1,\ldots,u_\ell,\,v_1,
\ldots,v_m,\,w_1,\ldots,w_n\in\mathcal A(H_\red)$ are such that
\[
\{v_1,\ldots,v_m\}\cap\{w_1,\ldots,w_n\} =\emptyset\,.
\]
Then $\gcd(z,z')=u_1\cdot\ldots\cdot u_{\ell}$, and we call
\[
\mathsf d (z,z') =\max\{m,\,n\} =\max\{ |z\gcd (z,z')^{-1}|,
|z'\gcd (z,z')^{-1}|\}\in\N_0
\]
the {\it distance} between $z$ and $z'$. If $\pi (z) =\pi (z')$ and
$z\ne z'$, then
\begin{equation} \label{dist-1}
2 +
      \bigl| |z |-|z'| \bigr| \le \mathsf d (z, z') \,,
\end{equation}
whence $|\mathsf Z (a)|\ge 2$ implies that  (details con be found in \cite[Lemma 1.6.2]{Ge-HK06a})
\begin{equation} \label{dist-2}
2 + \sup \Delta (\mathsf L(a)) \le \mathsf c (a) \,.
\end{equation}      
For an atom $u \in H$, let $\mathsf t (H,u)$ denote the smallest $N \in \N_0 \cup \{\infty\}$ with the following property:
\begin{itemize}
\item[] If $a \in uH$ and $z \in \mathsf Z (a)$, then there is a factorization $z' \in \mathsf Z (a)$, where  $uH^{\times}$ pops up, such that $\mathsf d (z,z') \le N$.
\end{itemize}
We say that $H$ is {\it locally tame} if $\mathsf t (H,u) < \infty$ for all $u \in \mathcal A (H)$. If $u$ is not prime, then $\omega (u) \le \mathsf t (H, u)$, where $\omega (u)$ denotes the omega-invariants studied in \cite{Ba23a, Ci-GS-LL25a}.

\medskip
\noindent
{\bf Chains of factorizations}. Let $a\in H$ and
$N\in\mathbb N_0$. A finite sequence $z_0,\ldots,
z_k\in\mathsf Z (a)$ is called a {\it $($monotone$)$ $N$-chain of
factorizations of $a$} if $\mathsf d (z_{i-1},z_i)\le N$ for all $i\in
[1,k]$ (and $|z_0|\le\ldots\le |z_k|$ or $|z_0|\ge\ldots\ge
|z_k|$). We denote by $\mathsf c (a)$ (or by $\mathsf c_{\mon} (a)$ resp.) the smallest $N\in\N _0\cup\{\infty\}$ such
that any two factorizations $z,\,z'\in\mathsf Z (a)$ can be
concatenated by an $N$-chain (or by a monotone $N$-chain resp.).
Then
\[
\mathsf c(H) =\sup\{\mathsf c(b)\colon b\in H\}\in\N_0\cup
\{\infty\}\quad\text{resp.}\quad\mathsf c_{\mon} (H) =\sup\{
\mathsf c_{\mon} (b)\colon b\in H\}\in\N_0\cup\{\infty\}\quad
\,
\]
denote the {\it catenary degree} resp. the {\it monotone
catenary degree} of $H$. By the Inequality \eqref{dist-2}, we see that $\mathsf c (H) \le 3$ implies that all length sets are intervals. Local tame degrees and catenary degrees are the main  invariants beyond length sets, which describe the arithmetic of atomic non-factorial monoids and domains.

\medskip
\noindent
{\bf Integral domains}. 
Let $D$ be an integral domain. We denote by  $D^{\bullet}$  its monoid of nonzero elements. Then $D$ is a Mori (or Krull or \G-) domain if and only if $D^{\bullet}$ is a Mori (or Krull or \G-) monoid (for details see \cite[Chapter 2]{Ge-HK06a}). The same correspondence holds true for weakly Krull domains and for C-domains, which will be introduced in later sections. All algebraic and arithmetic notation and properties, introduced for monoids, will be used for domains in the natural way. In particular, $\widehat D$ denotes the complete integral closure of the domain $D$ and $\mathsf c (D)$  its catenary degree, and so on.

\smallskip
\section{Algebraic properties of conductor submonoids} \label{3}
\smallskip

If $F$ is a monoid and $H \subset F$ is a submonoid, then $(H \DP F)$ is called the {\it conductor} of $H$ in $F$. We start with a simple well-known observation.

\begin{lemma} \label{3.1}
Let $F$ be a monoid and let $H \subset F$ be a submonoid with $(H \DP F) \ne \emptyset$. Then $F \subset \mathsf q (H)$ and $\widehat H = \widehat F$.
\end{lemma}

\begin{proof}
Obviously, we have $H \subset F \subset \mathsf q (F)$. If $x \in \mathsf q (F)$ and $f \in (H \DP F)$, then there are $a, b \in F$ with
\[
x = \frac{b}{a} = \frac{fb}{fa} \in \mathsf q (H) \,.
\]
Furthermore, we have
$\widehat H \subset \widehat F \subset \mathsf q (F)=\mathsf q (H)$. Conversely, let $x \in \widehat F$ and $c \in F$ such that $cx^n \in F$ for all $n \in \N$. Then for some $f \in (H \DP F)$ we obtain that $fcx^n \in H$ for all $n \in \N$, whence $x \in \widehat H$.
\end{proof}

\smallskip
\begin{lemma} \label{3.2}
Let $F$ be a factorial monoid and let $H \subset F$ be a submonoid with $H\neq H^{\times}$. Then the following are equivalent.
\begin{enumerate}
\item[(a)] $H = (H \DP F) \cup H^{\times}$.

\item[(b)] $H \setminus H^{\times}$ is an $s$-ideal of $F$.
\end{enumerate}
If the above conditions hold, then $H \subset F$ is a submonoid with $\widehat H = F$, $(H \DP F)$ is a $v$-ideal of $H$,  and $H^{\times} = F^{\times} \cap  H$.
\end{lemma}

\begin{proof}
If $H = F$, then all claims hold. Suppose that $H \subsetneq F$. If $\varepsilon \in F^{\times} \cap (H \DP F)$, then $F = \varepsilon (\varepsilon^{-1}F) = \varepsilon F \subset H$, a contradiction. Thus, it follows that $H^{\times} \cap (H \DP F) \subset F^{\times} \cap (H \DP F) = \emptyset$.

(a) $\Longrightarrow$ (b) We have
\[
(H \setminus H^{\times})F \subset (H \DP F)F \subset (H \DP F)  \subset H \,.
\]
Since $H^{\times} \cap (H \DP F) = \emptyset$, we infer that $(H \setminus H^{\times})F \subset H \setminus H^{\times}$, whence $H \setminus H^{\times}$ is an $s$-ideal of $F$.

(b) $\Longrightarrow$ (a) Since $(H \setminus H^{\times})F \subset H  \setminus H^{\times} \subset H$, it follows that $H \setminus H^{\times} \subset (H \DP F)$, whence $H = (H \DP F) \cup H^{\times}$.

Suppose that (a) and (b) hold. Since $1 \in H$ and
\[
H H = \big( (H \DP F) \cup H^{\times} \big) \big( (H \DP F) \cup H^{\times} \big) \subset (H \DP F) H \cup H^{\times}H^{\times} \subset H \,,
\]
it follows that $H$ is a submonoid of $F$. Lemma \ref{3.1} implies that $\widehat H = \widehat F$ and since $F$ is factorial we have $\widehat F = F$. Since $F^{\times} \cap (H \DP F) = \emptyset$, it follows from (a) that $F^{\times}\cap H\subset H^{\times}$ and hence $F^{\times}\cap H=H^{\times}$.
Finally, $(H \DP F)$ is a $v$-ideal of $H$ by \cite[Proposition 2.3.4]{Ge-HK06a}.
\end{proof}

\begin{definition} \label{3.3}
Let $F$ be a factorial monoid. A submonoid $H \subset F$ with $H\neq H^{\times}$  is called a {\it conductor submonoid} of $F$ if $H$ satisfies the equivalent conditions of Lemma \ref{3.2}.
\end{definition}

Let $F$ be a factorial domain and let $D \subset F$ be a subring. Then $D^{\bullet} = D \setminus \{0\}$ is a monoid,  $F^{\bullet} = F \setminus \{0\}$ is a factorial monoid, $D^{\times} = (D^{\bullet})^{\times}$, $F^{\times} = (F^{\bullet})^{\times}$, and
\[
(D \DP F) = \{f \in F \colon fF \subset D\} = \{f \in F^{\bullet} \colon f F^{\bullet} \subset D^{\bullet}\} \cup \{0\} = (D^{\bullet} \colon F^{\bullet}) \cup \{0\} \,.
\]
Thus,
\[
D = (D \DP F) \cup D^{\times} \quad \text{if and only if} \quad D^{\bullet} = (D^{\bullet} \DP F^{\bullet}) \cup D^{\times} \,,
\]
and if this holds, then $D$ is called a {\it conductor subring} of $F$. If $F$ is not a field and $D \subsetneq F$, then $D \setminus D^{\times} = (D \DP F)$ is a divisorial ideal, whence $D$ is a local domain.

\smallskip
\begin{lemma} \label{3.4}
Let $F$ be a factorial monoid and let $H \subset F$ be a conductor submonoid. 
\begin{enumerate}
\item  If $H \subsetneq F$, then the inclusion $H \hookrightarrow F$ is not a divisor homomorphism and $H$ is not a Krull monoid.

\item  $H$ is a \BF-monoid and it is an \FF-monoid if and only if $(F^{\times} \DP H^{\times}) < \infty$.
    
\item The reduced monoid $H_{\red} = H/H^{\times} = \{ a H^{\times} \colon a \in H\}$ is a conductor submonoid of $F/H^{\times} = \{aH^{\times} \colon a \in F\}$.

\item Let $s\in F$ and $a\in H\setminus H^{\times}$. If $sa\in \mathcal A(H)$, then $a\in \mathcal A(H)$.  
\end{enumerate}
\end{lemma}

\begin{proof}
1.
The inclusion is a divisor homomorphism if and only if $H = \mathsf q (H) \cap F$. But, we have $\mathsf q (H) \cap F = \mathsf q (F) \cap F = F$. If $H$ is Krull, then $H = \widehat H = F$ by Lemma \ref{3.2}.

2. Since $F$ is a BF-monoid and $H^{\times} = F^{\times} \cap H$ by Lemma \ref{3.2}, $H$ is a \BF-monoid by \cite[Corollary 1.3.3]{Ge-HK06a}. The characterization of the FF-monoid property follows from \cite[Theorem  1.5.6]{Ge-HK06a}.

3. Obvious.

4. By Definition \ref{3.3}, we have $sa\in H\setminus H^{\times}$.
Suppose $sa\in \mathcal A(H)$. Assume to the contrary that $a\not\in \mathcal A(H)$. Then $a=b_1b_2$ with $b_1,b_2\in H\setminus H^{\times}$, whence $sb_1\in H\setminus H^{\times}$ by Definition \ref{3.3}. It follows from $sa=(sb_1)b_2$ that $sa\not\in \mathcal A(H)$, a contradiction.
\end{proof}

Note that every Mori monoid is a BF-monoid and every BF-monoid satisfies the ACC on principal ideals. A conductor monoid $H$ is a submonoid of a factorial monoid $F$ such that $H^{\times} = F^{\times} \cap H$. C-monoids too are submonoids of factorial monoids with this property. We recall their definition.

\smallskip
Let $F$ be a factorial monoid and let $H \subset F$ be a submonoid. Two elements $y, y' \in F$ are called $H$-equivalent  if $y^{-1}H \cap F = {y'}^{-1} H \cap F$ or, in other words,
\[
\text{if for all} \ x \in F, \ \text{we have} \ xy \in H \quad \text{if and only if} \quad xy' \in H \,.
\]
$H$-equivalence defines a congruence relation on $F$ and for $y \in F$, let $[y]_H^F = [y]$ denote the $(H,F)$-congruence class of $y$. Then
\[
\mathcal C (H,F) = \{[y] \colon y \in F \} \quad \text{and} \quad \mathcal C ^* (H,F) = \{ [y] \colon y \in (F \setminus F^{\times}) \cup \{1\} \}
\]
are commutative semigroups with identity element $[1]$. $\mathcal C (H,F)$ is the {\it class semigroup} of $H$ in  $F$ and the subsemigroup $\mathcal C^* (H,F) \subset \mathcal C (H,F)$ is the {\it reduced class semigroup} of $H$ in $F$.
A monoid $H$ is called a C-{\it monoid} (defined in a factorial monoid $F$) if it is a submonoid of $F$ such that $H^{\times} = F^{\times} \cap H$ and the reduced class semigroup is finite. If $H$ is a C-monoid, then $H$ is Mori, $(H \DP \widehat H) \ne \emptyset$, and $\widehat H$ is Krull with finite class group (see \cite[Chapter 2]{Ge-HK06a}). A domain $D$ is a \C-domain if and only if $D^{\bullet}$ is a \C-monoid (\cite{Re13a, Ge-Ra-Re15c, Ba-Po25a}).

\smallskip
\begin{theorem} \label{3.5}
Let $F$ be a factorial monoid and let $H \subset F$ be a conductor submonoid. Then $H$ is a Mori \G-monoid if and only if
$F_{\red}$ is finitely generated. If this holds, then $H$ is a $v$-local \C-monoid defined in $F$ and $s$-$\spec (H)$ is finite.
\end{theorem}

\begin{proof}
	Since $H\subset F$ is a conductor submonoid, it follows from Lemma \ref{3.2} that $\widehat{H}=F$ and  $H\setminus H^{\times}=(H:F)\neq \emptyset$ is a $v$-ideal. By \cite[Theorem 2.2.5]{Ge-HK06a}, we obtain $H$ is $v$-local.

 If $H$ is a Mori \G-monoid, then by 
\cite[Theorem 2.7.9.1]{Ge-HK06a}, we obtain $F_{\red}=\widehat{H}_{\red}$ is finitely generated.
Conversely, suppose that $F_{\red}$ is finitely generated, say $F = F^{\times} \times \mathcal F ( \{p_1, \ldots, p_s\})$. We claim that $H$ is a C-monoid defined in $F$. Then $H$ is Mori (because all C-monoids are Mori) and it follows from \cite[Theorem 2.9.15.3]{Ge-HK06a} that $H$ is a \G-moniod and $s$-$\spec (H)$ is finite. Thus, 
by \cite[Theorem 2.9.7]{Ge-HK06a}, we need to show that there is an $\alpha \in \N$ such that for all $j \in [1,s]$ and all $a \in p_j^{\alpha}F$ we have
\[
a \in H \quad \text{if and only if} \quad p_j^{\alpha}a \in H \,.
\]
Let $j \in [1,s]$, $\alpha \in \N$ arbitrary,  and $a \in p_j^{\alpha}F$. If $a \in H$, then $a \not\in H^{\times}$, whence $p_j^{\alpha} a \in (H \setminus H^{\times})F \subset H \setminus H^{\times}$.

Conversely, we start with the definition of  a suitable $\alpha$. We have $(H \setminus H^{\times})F \subset H \setminus H^{\times}$. Thus, if $b = \varepsilon p_1^{k_1} \cdot \ldots \cdot p_s^{k_s} \in H \setminus H^{\times}$, then $\varepsilon' p_1^{k_1'} \cdot \ldots \cdot p_s^{k_s'} \in H \setminus H^{\times}$ for all $\varepsilon' \in F^{\times}$ and all $k_i' \in \N_{\ge k_i}$. We set
\[
M = \{(k_1, \ldots, k_s) \in \N_0^s \colon F^{\times} p_1^{k_1} \cdot \ldots \cdot p_s^{k_s} \subset H \setminus H^{\times} \}
\]
and observe that the set of minimal points $\text{\rm Min} (M)$ is finite by Dickson's Lemma  (\cite[Theorem 1.5.3]{Ge-HK06a}). We define
\[
\alpha = \max \{k_i \colon i \in [1,s], (k_1, \ldots, k_s) \in \text{\rm Min} (M) \}  \,.
\]
Now consider an element $a = p_j^{\alpha}F$ with $p_j^{\alpha}a \in H$ for some $j \in [1,s]$, say $j=1$. Since $p_1^{\alpha}a \in H$, there is $(k_1^*, \ldots, k_s^*) \in \text{\rm Min} (M)$ with $k_1^* \le \alpha + \mathsf v_{p_1}(a)$ and with $k_i^* \le \mathsf v_{p_i} (a)$ for all $i \in [2,s]$. Since $\mathsf v_{p_1} (a) \ge \alpha \ge k_1^*$, it follows that $k_i^* \le \mathsf v_{p_i} (a)$ for all $i \in [1,s]$, whence $a \in H$.
\end{proof}

\smallskip
\begin{definition} \label{3.6}
Let $F$ be a factorial monoid and let $H \subset F$ be a submonoid with $H^{\times}=F^{\times}\cap H$ and $H\neq H^{\times}$.
Then  \begin{enumerate}
	\item $\mathcal G (H) = F \setminus H$ is called the {\it gap set} of $H$.
	
	\item $H$ is called {\it gap absorbing} if 
	\[
	\mathcal G (H)   \mathcal A (H)    \subset \mathcal A (H) \cup \mathcal A (H)  \mathcal A (H) \,.
	\]
	
	\item An element $a\in H\setminus H^{\times}$ is called {\it minimal} if for every $b\in H\setminus H^{\times}$ with $b\mid_F a$, we have $b\in aH^{\times}$.
\end{enumerate} 
\end{definition}

We denote by $\mathcal M(H)$ the set of all minimal elements of $H$. Then $\mathcal M(H)\subset \mathcal A(H)$.

\begin{lemma}\label{minimal}
	Let $F$ be a factorial monoid and let $H \subset F$ be a submonoid with $H^{\times}=F^{\times}\cap H$ and $H\neq H^{\times}$.
	Then for every $a\in \mathcal A(H)\setminus \mathcal M(H)$, there exists $s\in \mathcal G(H)$ such that $s^{-1}a\in \mathcal M(H)$.
\end{lemma}

\begin{proof}
 Let 	$a\in \mathcal A(H)\setminus \mathcal M(H)$. Let $b\in H\setminus H^{\times}$ be a divisor of $a$ in $F$ with minimal length.
Then $b\in \mathcal M(H)\subset \mathcal A(H)$. If $b^{-1}a\in H^{\times}$, then $a\in bH^{\times}\subset \mathcal M(H)$, a contradiction. Thus $b^{-1}a\not\in H^{\times}$. It follows from $a\in \mathcal A(H)$ that  $b^{-1}a\not\in H\setminus H^{\times}$, whence $s:=b^{-1}a\in F\setminus H=\mathcal G(H)$.	
\end{proof}

The following lemma shows that  the previous definition coincides with the definition of gap absorbing monoids given by Cisto,  Garc\' ia-S\' anchez, and Llena \cite[Section 2]{Ci-GS-LL25a}, in the setting of  free abelian monoids.

\smallskip
\begin{lemma}\label{3.7}
	Let $F$ be a factorial monoid and let $H \subset F$ be a gap absorbing submonoid. 
Then $H$ is a conductor submonoid  and  $$\mathcal G (H)  \mathcal G (H) \ \subset \ \mathcal G (H) \cup H^{\times}\cup \mathcal A (H) \cup \mathcal A (H)  \mathcal A (H)\,.$$
\end{lemma}

\begin{proof}
 First, we show that $H\setminus H^{\times}$ is an $s$-ideal of $F$, which implies that $H$ is a conductor monoid. Let $a\in H\setminus H^{\times}$ and $f\in F$. Then $af\not\in H^{\times}$.  Since $a\in H\setminus H^{\times}\subset F\setminus F^{\times}$, there exists $b\in \mathcal A(H)$ and $c\in H$ such that $a=bc$. If $f\in H$, then $af\in H$. If $f\in \mathcal G(H)$, then  Definition \ref{3.6} implies that $bf\in H$ and hence $af=cbf\in H$.
	
Next, we show the asserted inclusion of the product set $\mathcal G (H)  \mathcal G (H)$.
Let $s_1,s_2\in \mathcal G(H)$ such that $s_1s_2\in \mathcal G (H)  \mathcal G (H) \setminus \mathcal G (H)\subset H$.
	 If $s_1,s_2\in F^{\times}\setminus H^{\times}$, then $s_1s_2\in F^{\times}\cap H=H^{\times}$.
Otherwise there exists $j\in [1,2]$, say $j=1$, such that $s_1\not\in F^{\times}$ and hence $|s_1|_F\ge 1$. It follows $s_1s_2\in H\setminus H^{\times}$, whence there exist an atom $b\in \mathcal A(H)$ and $c\in H$ such that $bc=s_1s_2$.
We can decompose $s_1s_2=s_3s_4$, with $s_3,s_4\in \mathcal G(H)\cup F^{\times}$, such that $s_3$ has maximal length in $F$.
	 If $s_4\in F^{\times}$, then $s_3\not\in F^{\times}$ and  $s_3=s_4^{-1}s_1s_2=s_4^{-1}cb\in bF\subset H$ since $H\setminus H^{\times}$ is an $s$-ideal of $F$, a contradiction.
	 Thus $s_4\in \mathcal G(H)\setminus F^{\times}$ and hence $|s_4|_F\ge 1$. Let $p\in F$ be prime such that $p\mid_F s_4$. Since $H\setminus H^{\times}$ is an $s$-ideal of $F$, it follows from $s_4\in \mathcal G(H)$ that $p^{-1}s_4\in \mathcal G(H)\cup F^{\times}$.  
	 Since $s_3\in \mathcal G(H)$,  we have that $ps_3\in \mathcal G(H)\cup \mathcal A(H)$ with $|ps_3|_F>|s_3|_F$, whence the maximality of $|s_3|_F$ implies that $ps_3\in \mathcal A(H)$. Now $s_1s_2=(p^{-1}s_4)(ps_3)$ and the assertion follows by Definition \ref{3.6}.
\end{proof}

Our next theorem and its corollary answer Conjectures \cite[Conjecture 2.3]{Ci-GS-LL25a}, \cite[Conjecture 8.4]{Ci-GS-LL25a}, and \cite[Conjecture 4.16]{Ba23a} in the affirmative.

\begin{theorem} \label{3.8}
Let $F$ be a factorial monoid and let $H \subset F$ be a submonoid. Then $H$ is a conductor submonoid if and only if it is gap absorbing.
\end{theorem}

\begin{proof}
If $H$ is gap absorbing, then $H$ is a conductor submonoid by Lemma \ref{3.7}. Conversely, suppose that 
$H \subset F = F^{\times} \times \mathcal F (P)$ is a conductor submonoid, whence $(H\setminus H^{\times})F=H\setminus H^{\times}$ and hence $\mathcal G (H)   \mathcal A (H)\subset H\setminus H^{\times}$. 
 We need to show 
\[ 
 \mathcal G (H)   \mathcal A (H)    \setminus\mathcal A (H) \subset \mathcal A (H)  \mathcal A (H) \,.
\] 
Let $s_1\in \mathcal G (H)$ and $a^*\in \mathcal A (H)$ be such that $s_1a^*\in H\setminus \mathcal A (H)$. 
Then $s_1a^*\not\in H^{\times}$, $s_1\not\in F^{\times}$, and Lemma \ref{minimal} implies that there exist $s_2\in \mathcal G(H)\cup \{1\}$ and $a\in \mathcal M(H)$ such that  $a^*=as_2$.
 
Suppose $s_1s_2\in H$. We choose a divisor $s$ of $s_2$ in $F$ such that $s_1s\in H$ with $|s|_F$ minimal. Since $s_1\not\in F^{\times}$, we have $s_1s\in H\setminus H^{\times}$. If $s_1s\not\in \mathcal A(H)$, then $s_1s=x_1x_2$ for some $x_1,x_2\in H\setminus H^{\times}$, which implies that there are decompositions 
\[
s_1=b\cdot (b^{-1}s_1) \text{ and } s=c\cdot (c^{-1}s), \text{ where }b, b^{-1}s_1, c, c^{-1}s\in F \text{ such that } c(b^{-1}s_1), b(c^{-1}s)\in H\setminus H^{\times}\,.
\]
Thus $s_1c=bc(b^{-1}s_1)\in c(b^{-1}s_1)F\subset H$. By the minimality of $|s|_F$, we obtain that $c^{-1}s\in F^{\times}$ and hence $b=(s^{-1}c)b(c^{-1}s)\in F^{\times}(H\setminus H^{\times})\subset H\setminus H^{\times}$, which implies that $s_1=b\cdot (b^{-1}s_1)\in bF\subset H$, a contradiction.  Thus $s_1s\in \mathcal A(H)$. Since $s^{-1}a^*\in \mathcal A(H)$ by Lemma \ref{3.4}.4, it follows that $s_1a^*=(s_1s)(s^{-1}a^*)\in \mathcal A (H)  \mathcal A (H)$.

Suppose $s:=s_1s_2\in \mathcal G(H)$. Since $sa=s_1a^*\not\in \mathcal A(H)$, we have $sa=a_1a_2$ for some $a_1,a_2\in H\setminus H^{\times}$, which implies that there are decompositions 
 \[
 s=b\cdot (b^{-1}s) \text{ and } a=c\cdot (c^{-1}a), \text{ where }b, b^{-1}s, c, c^{-1}a\in F \text{ such that } c(b^{-1}s), b(c^{-1}a)\in H\setminus H^{\times}\,.
 \]
 We may choose the pair $(b,c)$ satisfying the above property such that $|b|_F+|c|_F$ is minimal. If $c\in F^{\times}$, then $b^{-1}s\in H\setminus H^{\times}$ and hence $s=b(b^{-1}s)\in b^{-1}sF\subset H$, a contradiction. Thus $c\not\in F^{\times}$ and since $a\in \mathcal M(H)$, we have
 \begin{equation}\label{***}
   c^{-1}a\in \mathcal G(H)\,.
 \end{equation}
 
 If $c(b^{-1}s)\not\in \mathcal A(H)$, then $c(b^{-1}s)=h_1h_2$ for some $h_1,h_2\in H\setminus H^{\times}$, which implies that $c=c_1c_2$ and $b^{-1}s=d_1d_2$ with $h_1=c_1d_1$ and $h_2=c_2d_2$, where $c_1,c_2,d_1,d_2\in F$. Thus 
 $$c_1(b^{-1}s)=c_1d_1d_2=d_2h_1\in h_1F\subset H\setminus H^{\times}\,.$$ 
 Moreover, $b(c^{-1}a)\in H\setminus H^{\times}$ implies that $$b(c_1^{-1}a)=c_2b(c^{-1}a)\in b(c^{-1}a)F\subset H\setminus H^{\times}\,.$$
 Now we have decompositions
 \[
 s=b\cdot (b^{-1}s) \text{ and } a=c_1\cdot (c_1^{-1}a), \text{ where }b, b^{-1}s, c_1, c_1^{-1}a\in F \text{ such that } c_1(b^{-1}s), b(c_1^{-1}a)\in H\setminus H^{\times}\,.
 \]
 Since $c_2\in F^{\times}$ implies that  $s=bd_1d_2=c_2^{-1}bd_1h_2\in h_2F\subset H$, it follows from \eqref{***}  that $c_2\notin F^{\times}$, whence
 $|c_1|_F<|c|_F$.  Therefore $|b|_F+|c_1|_F<|b|_F+|c|_F$, a contradiction to our choice of $(b,c)$. Thus $c(b^{-1}s)\in \mathcal A(H)$.
 
  If $b(c^{-1}a)\not\in \mathcal A(H)$, then $b(c^{-1}a)=h_1h_2$ for some $h_1,h_2\in H\setminus H^{\times}$, which implies that $b=b_1b_2$ and $c^{-1}a=d_1d_2$ with $h_1=b_1d_1$ and $h_2=b_2d_2$, where $b_1,b_2,d_1,d_2\in F$. Thus $$b_1(c^{-1}a)=b_1d_1d_2=d_2h_1\in h_1F\subset H\setminus H^{\times}\,.$$
   Moreover, $c(b^{-1}s)\in H\setminus H^{\times}$ implies that $$c(b_1^{-1}s)=b_2c(b^{-1}s)\in c(b^{-1}s)F\subset H\setminus H^{\times}\,.$$
 Since $b_2\in F^{\times}$ implies $c^{-1}a=d_1d_2=b_2^{-1}d_1h_2\in h_2F\subset H$, it follows from \eqref{***}  that  $b_2\notin F^{\times}$, whence
 $|b_1|_F<|b|_F$.  Therefore $|b_1|_F+|c|_F<|b|_F+|c|_F$, a contradiction to our choice of $(b,c)$. Thus $b(c^{-1}a)\in \mathcal A(H)$, whence $sa=s_1a^*=(cb^{-1}s)(bc^{-1}a)\in \mathcal A (H)  \mathcal A (H)$.
\end{proof}

\smallskip
We continue with a discussion of examples. In addition to examples already given in the literature (e.g., \cite{Ba23a, Ci25a, Ci-GS-LL25a}), we also provide examples from ring theory.  In particular, we want to point out that conductor submonoids need neither be seminormal nor finitely generated. Examples of a very different kind will be given in Section \ref{4}.

\smallskip
\begin{example}~ \label{3.9}

1. Let $H \subset F = \mathcal F (P)$ be a conductor submonoid such that $F \setminus H$ is finite. Conductor submonoids of this type are studied in detail in \cite{Ba23a}. In particular, if $P$ is finite, then $H$ is finitely generated by \cite[Proposition 3.3]{Ba23a}, and these monoids are a special class of the submonoids studied by Gotti in \cite{Go20a,Go20c}.
The case when $P$ and $F \setminus H$ are finite is studied in detail in \cite{Ci25a}. To be more explicit, we switch to additive notation. Suppose that $F = (\N_0^s, +)$ with $s \in \N$ and let $\boldsymbol e_1, \ldots, \boldsymbol e_s$ denote the canonical basis vectors. A submonoid $H \subset F$ is called a {\it generalized numerical monoid} if $F \setminus H$ is finite. Choose incomparable elements $u_1, \ldots, u_t \in \N_0^s$ and let $\mathfrak a = \bigcup_{i=1}^t (u_i + \N_0^s)$ denote the $s$-ideal generated by $u_1, \ldots, u_t$. If for each $i \in [1,s]$ there is $m_i \in \N$ such that $m_i \boldsymbol e_i \in \mathfrak a$, then $H = \mathfrak a \cup \{\boldsymbol 0\}$ is a conductor submonoid and a generalized numerical monoid.

2. A monoid $H$ is called {\it finitely primary} if there exist $s, \alpha \in \N$ such that $H$ is the submonoid of a factorial monoid $F = F^{\times} \times \mathcal F ( \{p_1, \ldots, p_s\})$ satisfying
\[
H \setminus H^{\times} \subset p_1 \cdot \ldots \cdot p_sF \quad \text{and} \quad (p_1 \cdot \ldots \cdot p_s)^{\alpha}F \subset H \,.
\]
Finitely primary monoids need not be Mori (see \cite{HK-Ha-Ka04}), but they are primary whence they are \G-monoids. Clearly, numerical monoids are finitely primary of rank one. We discuss two special situations.

(i) If $H$ is as above, then its seminormalization $H'$ is a C-monoid and it has the form
\[
H' = p_1 \cdot \ldots \cdot p_sF \cup H'^{\times}
\] 
(\cite[Lemma 3.4]{Ge-Ka-Re15a}). Clearly, $H'$ is a conductor submonoid of $F$. If $s=1$, then $\mathsf c (H') \le 2$ and $H'$ is half-factorial (\cite[Lemma 3.6]{Ge-Ka-Re15a}), whence there are half-factorial conductor submonoids.

(ii) If $H$ has  the form $H = (p_1^{\alpha_1} \cdot \ldots \cdot p_s^{\alpha_s})F \cup H^{\times}$, then it  is also a conductor submonoid of $F$. However, 
if $s \ge 2$ and all $\alpha_i \ge 2$, then $H$ is neither seminormal nor finitely generated.

\smallskip
3. We already observed that conductor subrings of factorial domains are local domains. We analyze the assumptions of Theorem \ref{3.8} for domains. Suppose that $D$ is Mori with $(D \DP \widehat D) \ne \{0\}$.  Then $D$ is one-dimensional semilocal if and only if $D$ is a G-domain (\cite[Proposition 2.10.7]{Ge-HK06a}) and if this holds, then $\widehat D$ is a semilocal principal ideal domain. 

If $D$ is a one-dimensional local Mori domain with nonzero conductor $(D \DP \widehat D) \ne \{0\}$, then $D^{\bullet}$ is finitely primary and its rank equals the number of maximal ideals of $\widehat D$ (\cite[Proposition 2.10.7]{Ge-HK06a}). Thus, Part 2 provides a variety of examples of conductor subrings. In particular, $D^{\bullet}$ is seminormal finitely primary (and hence a conductor monoid) if and only if $D$ is a seminormal one-dimensional local Mori domain (\cite[Lemma 3.4]{Ge-Ka-Re15a}).

Faber \cite[Proposition 4.2]{Fa20a} studied reduced one-dimensional local noetherian rings $(D, \mathfrak m)$. In case of domains her characterization states that $D$ is a conductor domain if and only if $D$ is $1$-step normal (i.e., $\widehat  D = \End_D ( \mathfrak m)$).

To give an explicit example, we consider power series rings. Let  $K \subsetneq L$ be a field extension,
let $D=K+X^nL[[X]]$ with $n \in \N$, let  $H=D^\bullet$, $\widehat D = L[[X]]$, and $F=\widehat D^\bullet$. Then  $H\setminus H^\times=(X^nL[[X]])^\bullet$ is an  $s$-ideal of the factorial monoid $F$, whence $D$ is a conductor subring of $\widehat D$.
\end{example}

\smallskip
Next we construct, for every $n \in \N$, a local noetherian conductor domain with Krull dimension $n$. 
This  requires some preliminaries.  Given two ring homomorphisms $f:A\to C$ and $g:B\to C$, the fiber product of $f$ and $g$ (in the category of rings) is the subring
$$
f\times_Cg:=\{(a,b)\in A\times B\colon f(a)=g(b)\}
$$
of the direct product $A\times B$.  This kind of construction is usually very helpful to provide examples of rings with prescribed prime spectrum or, more generally, with certain ideal-theoretic properties.  If $\lambda:R\to S$ is a ring homomorphism, then $\lambda^\star:\spec(S)\to\spec(R)$ is the canonical continuous map (with respect to the Zariski topology) induced by $\lambda$, defined by $\lambda^\star(\f q) =\lambda^{-1}(\f q)$ for every $\f q\in\spec(S)$. The  closed subspaces of $\spec(R)$, with respect to the Zariski topology, are the sets of the type 
\[
\mathcal V_R(\f a):=\{\f p\in\spec(R)\colon \f p \supset \f a \},
\]
for every ideal $\f a$ of $R$. The following lemma summarizes  some basic facts about the prime ideal structure of a fiber product (see \cite[Theorem 1.4 and its proof]{Fo80}).

\smallskip
\begin{lemma} \label{3.10} 
Let $f$ and $g$ be as above, suppose that $g$ is surjective, and let  $D:=f\times_Cg$ denote their fiber product. 
\begin{enumerate}
\item Let $p:D\to A$ be the restriction to $D$ of the projection $A\times B\to A$. Then $p^\star$ is a closed embedding inducing a homeomorphism of $\spec(A)$ onto $\mathcal V_D(\Ker(p))$. 

\item Let $q:D\to B$ be the restriction to $D$ of the projection $A\times B\to B$. Then $q^\star$ restricts to a homeomorphism  $\spec(B)\setminus \mathcal V_B(\Ker(g))\to \spec(D)\setminus \mathcal V_D(\Ker(p))$.  
\end{enumerate}
In particular, $\spec(D)$ is the disjoint union of $p^\star \big(\spec(A) \big)$ and of $q^\star \big(\spec(B)\setminus \mathcal V_B(\Ker(g)) \big)$. 
\end{lemma}

\smallskip
\begin{example}~ \label{3.11}
Let $K \subsetneq L$ be a finite field extension and let $L[T_1, \ldots, T_n]$ be the polynomial ring with $n \ge 1$ variables. Then $\mathfrak n= \langle T_1, \ldots, T_n\rangle$ is a maximal ideal of $L[T_1, \ldots, T_n]$, and we consider the localization $F = L[T_1, \ldots, T_n]_{\mathfrak n}$, which is a local factorial domain with maximal ideal $\mathfrak m = \mathfrak n_{\mathfrak n}$. Then the residue field of $F$ is isomorphic to $L$ and we consider the canonical projection
\[
p \colon F \to L \,, \quad \text{defined by} \quad \frac{f}{g} \mapsto  \frac{f(0, \ldots, 0)}{g (0, \ldots, , 0)} \,.
\]
We assert that 
\[
D := p^{-1} (K) \subset F
\] 
is a conductor domain which is local noetherian and has Krull dimension $\dim (D)=n$.

To begin with, note that $D$ is canonically isomorphic to the fiber product of the inclusion $i:K\hookrightarrow L$ and of the projection $p$ (indeed,  the map $j:D\to i\times_Lp$, $\eta\mapsto (p(\eta),\eta)$ is an isomorphism). By Lemma \ref{3.10}, the following facts are straightforward. 
\begin{enumerate}
	\item[(i)] If $p_0:D\to K$ is the restriction of $p$ to $D$,  then $p_0^\star$ is the trivial closed embedding sending $(0)\in\spec(K)$ to $\f m$. In particular, $\f m$ is a maximal ideal of $D$. Since then $\f m$ is a common ideal of $D$ and $F$ and $D\neq F$ (because $K$ is a proper subfield of $L$), it immediately follows that $\f m$ is the conductor of the ring extension $D\subset F$. Thus, $D$ is a conductor subring of $F$.

\item[(ii)] Let $\iota \colon D \hookrightarrow F$ be the inclusion. Then $\iota^\star$ restricts to a homeomorphism of 
      \[
      \spec(F)\setminus \mathcal V_F(\f m)=\{\f h\in\spec(F)\colon \f h\subsetneq \f m \}
      \]
	  onto the subspace $\spec(D)\setminus \mathcal V_D(\f m)$. Since $\iota$ is the inclusion,  $\iota^\star$ is the contraction map. On the other hand, if $\f h\in \spec(F)\setminus \mathcal V_F(\f m)$, we have obviously $\f h\subsetneq D$, since we already know that $\f m\subset D$, and thus $\iota^\star(\f h)=\f h\cap D=\f h$, that is, $\iota^\star$ is the identity. It immediately follows that $\spec(D)$ and $\spec(F)$ are set theoretically equal whence, in particular, $\dim(D)=\dim(F)=n$.

\item[(iii)]  Since $K\subsetneq L$ is a finite field extension,  \cite[Corollary 1.5(4)]{Fo80}  implies that $D\subset F$ is a finite extension of rings. Since $F$ is noetherian, the Theorem of Eakin-Nagata implies that $D$ is noetherian. 
\end{enumerate}    
\end{example}

\medskip
\section{Conductor submonoids associated with Krull monoids} \label{4}
\medskip

In this section we study a class of conductor submonoids which was first considered by Baeth \cite{Ba23a}. A focus of his research was on arithmetic invariants (such as the $\rho_k (\cdot )$ invariants, the set of distances, and the $\omega (\cdot )$ invariants) of the monoid $F_{\iota} (G)$ (see Definition \ref{4.2}). Part of the former results are refined by Theorem \ref{5.1} and new results are given in Theorem \ref{5.5}.

A monoid homomorphism $\theta \colon H \to B$ is said to be a {\it transfer homomorphism} if the following two conditions are satisfied.
\begin{enumerate}
\item[{\bf (T1)\,}] $B = \theta(H) B^\times$  and  $\theta^{-1} (B^\times) = H^\times$.

\item[{\bf (T2)\,}] If $u \in H$, \ $b,\,c \in B$  and  $\theta (u) = bc$, then there exist \ $v,\,w \in H$ such that  $u = vw$, \  $\theta (v) \in bB^{\times}$, and  $\theta (w) \in c B^{\times}$.
\end{enumerate}
Every transfer homomorphism $\theta \colon H \to B$ between atomic monoids $H$ and $B$ gives rise to a unique homomorphism $\overline \theta \colon \mathsf Z (H) \to \mathsf Z (B)$ satisfying $\overline \theta (uH^{\times}) = \theta (u) B^{\times}$ for all $u \in \mathcal A (H)$ (see \cite[Proposition 3.2.3]{Ge-HK06a}).

For $a \in H$, we denote by  $\mathsf c (a, \theta)$  the
smallest $N \in \N_0 \cup \{\infty\}$ with the following property:

\smallskip

\begin{enumerate}
\item[]
If $z,\, z' \in \mathsf Z_H (a)$ and $\overline \theta (z) =
\overline \theta (z')$, then there exist some $k \in \N_0$ and factorizations $z=z_0,
\ldots, z_k=z' \in \mathsf Z_H (a)$ such that \ $\overline \theta (z_i) = \overline \theta (z)$ and \ $\mathsf d (z_{i-1}, z_i) \le N$ for all $i \in [1,k]$ \ (that is,
$z$ and $z'$ can be concatenated by an $N$-chain in the fiber
\ $\mathsf Z_H (a) \cap \overline \theta ^{-1} (\overline \theta
(z)$)\,).
\end{enumerate}
Then 
\[
\mathsf c (H, \theta) = \sup \{\mathsf c (a, \theta) \colon a \in H \} \in \N_0 \cup \{\infty\}
\]
denotes the {\it catenary degree in the fibres}.

A transfer homomorphism  allows to pull back  arithmetic properties from $B$ to $H$. We gather two key properties in the next lemma (see \cite[Theorem 3.2.5 and Lemma 3.2.6]{Ge-HK06a}).

\smallskip
\begin{lemma} \label{4.1}
Let $\theta \colon H \to B$ be a transfer homomorphism of atomic monoids and let $a \in H$.
\begin{enumerate}
\item $\mathsf L_H (a) = \mathsf L_B \big( \theta (a) \big)$ and $\mathcal L (H) = \mathcal L (B)$.

\item $\mathsf c \bigl(\theta(a)\bigr) \le \mathsf c(a) \le \max \{\mathsf c\bigl( \theta(a) \bigr),\,
      \mathsf c(a,\theta)\}$,  $\mathsf c(B) \le \mathsf c(H) \le \max \{\mathsf c(B), \mathsf c(H, \theta)\}$, and  \newline $\mathsf c_{\mon} (B) \le \mathsf c_{\mon} (H) \le \max \{\mathsf c_{\mon} (B), \mathsf c(H, \theta)\}$.
\end{enumerate}
\end{lemma}

\smallskip
A monoid homomorphism $\varphi \colon H \to F = \mathcal F (P)$ is said to be a
\begin{itemize}
\item {\it divisor homomorphism} if, for all $a, b \in H$ we have $a \mid b$ if and only if $\varphi (a) \mid \varphi (b)$, and

\item {\it divisor theory} if it is a divisor homomorphism and, for all $\alpha \in F$, there is a finite nonempty set $A \subset H$ such that $\alpha = \gcd \varphi (A)$. 
\end{itemize}
A monoid is Krull if and only if it has a divisor theory (equivalently, a divisor homomorphism to a free abelian monoid), and its divisor theory gives rise to a canonical transfer homomorphism. In order to describe it, we introduce  monoids of zero-sum sequences.

\smallskip
Let $G$ be an additive abelian group and let $G_0 \subset G$ be a subset. The elements of the (multiplicatively written) free abelian monoid $\mathcal F (G_0)$ with basis $G_0$ are called {\it sequences} over $G_0$. If
\[
S = g_1 \cdot \ldots \cdot g_{\ell} = \prod_{g \in G_0} g^{\mathsf v_g (S)} \ \in \mathcal F (G_0) \,,
\]
with $\ell \in \N_0$ and $g_1, \ldots, g_{\ell} \in G_0$, is a sequence over $G_0$, then
\begin{itemize}
\item $|S|=\ell = \sum_{g \in G_0} \mathsf v_g (S) \in \N_0$ is the {\it length} of $S$, and

\item $\sigma (S) = g_1 + \ldots + g_{\ell} = \sum_{g \in G_0} \mathsf v_g (S) g \in G$ is the {\it sum} of $S$.
\end{itemize}
We say that $S$ is {\it zero-sum free} if $\sum_{i \in I} g_i \ne 0$ for all $\emptyset \ne I \subset [1, \ell]$.
Then
\[
\mathcal B (G_0) = \{S\in \mathcal F (G_0) \colon \sigma (S) = 0 \} \subset \mathcal F (G_0)
\]
denotes the {\it monoid of zero-sum sequences} over $G_0$. Since the inclusion  $\iota \colon \mathcal B (G_0) \hookrightarrow \mathcal F (G_0)$ is a divisor homomorphism (but not necessarily a divisor theory), $\mathcal B (G_0)$ is a Krull monoid, and
\[
\mathsf D (G_0) = \sup \{ |S| \colon S \ \text{is an atom in} \ \mathcal B (G_0) \} \in \N_0 \cup \{\infty\}
\]
denotes the {\it Davenport constant} of $G_0$.

\smallskip
Now, let $\varphi \colon H \to F = \mathcal F (P)$ be a divisor theory
with class group $G$ and let $G_0 \subset G$ denote the set of classes containing prime divisors. Consider the following diagram
\[
\begin{CD}
H  @>{\varphi}>>    \mathcal F (P) \\
 @V{\boldsymbol \beta} VV                           @VV{\widetilde{\boldsymbol \beta}}  V\\
\mathcal B (G_0)           @>{\iota}>>         \mathcal F (G_0) \,,
\end{CD}
\]
where 
\begin{itemize}
\item $\widetilde{\boldsymbol \beta} \colon \mathcal F (P) \to \mathcal F (G_0)$ is the unique epimorphism mapping each $p \in P$ onto its class $[p] \in G_0$, and

\item $\boldsymbol \beta = \widetilde{\boldsymbol \beta} \circ \varphi \colon H \to \mathcal B (G_0)$.
\end{itemize}    
Then $\widetilde{\boldsymbol \beta}$ and $\boldsymbol \beta $ are transfer homomorphisms by \cite[Theorem 3.4.10]{Ge-HK06a}.

\smallskip
\begin{definition} \label{4.2}
Let $\varphi \colon H \to F = \mathcal F (P)$ be a divisor theory
with class group $G$, let $G_0 \subset G$ denote the set of classes containing prime divisors, and let $\iota \colon \mathcal B (G_0) \hookrightarrow \mathcal F (G_0)$ denote the inclusion map. Then
\begin{itemize}
\item $F_{\varphi} = \mathcal F_{\varphi} (P) = \{ \alpha \in F \colon \text{there is some $a \in H \setminus H^{\times}$ with $\varphi (a)$ divides $\alpha$} \} \cup \{1\} = \varphi (H \setminus H^{\times}) F \cup \{1\}$, resp.

\item $\mathcal F_{\iota} (G_0) = \{ S \in \mathcal F (G_0) \colon S \ \text{is not zero-sum free} \} \cup \{1\} = (\mathcal B (G_0) \setminus \{1\})\mathcal F (G_0) \cup \{1\}$,
\end{itemize}
are called the {\it conductor submonoids of $F$ associated with $\varphi$} resp. {\it the conductor submonoid of $\mathcal F (G_0)$ associated with $\iota \colon \mathcal B (G_0) \hookrightarrow \mathcal F (G_0)$}.
\end{definition}

By definition, we have $\varphi (H) \subset F_{\varphi}$ and $\mathcal B (G_0) \subset \mathcal F_{\iota} (G_0)$, and the inclusions
\[
\varphi (H) \hookrightarrow  F_{\varphi} \quad \text{and} \quad \mathcal B (G_0) \hookrightarrow \mathcal F_{\iota} (G_0)
\]
are divisor homomorphisms. Furthermore, 
note that $H$ is factorial if and only if  $|G|=1$. If this holds, then  $F_{\varphi} = F$ is factorial and $F_{\iota} (G) = \mathcal B (G) \cong (\N_0, +)$ (the converse is also true as we will see in Theorem \ref{5.5}).

\smallskip
\begin{proposition} \label{4.3}
Let $H$ be a Krull monoid with divisor theory $\varphi \colon H \to F$ and let $\mathcal I_v^* (H)$ be the monoid of $v$-invertible $v$-ideals, equipped with $v$-multiplication. Then there is a monoid isomorphism
\[
\Phi \colon F_{\varphi} \longrightarrow \{ I \in \mathcal I_v^* (H) \colon I \ \text{is divisible by a proper principal ideal} \} \subset  \mathcal I_v^* (H) \,,
\]
whence $F_{\varphi}$ is $($isomorphic to$)$ a monoid of $v$-invertible $v$-ideals lying between the monoid of all prinicipal ideals and the monoid of all $v$-invertible $v$-ideals. In particular, $F_{\varphi}$ is $($up to canonical isomorphism$)$ uniquely determined by $H$.
\end{proposition}

\begin{proof}
Since $H$ is a Krull monoid, the monoid $\mathcal I_v^* (H)$ is free abelian with basis $\mathfrak X (H)$ and the map $\partial \colon H \to \mathcal I_v^* (H)$, defined by $\partial (a) = aH$ for all $a \in H$, is a divisor theory (\cite[Proposition 2.4.5]{Ge-HK06a}). By the Uniqueness Theorem for Divisor Theories (\cite[Theorem 2.4.7]{Ge-HK06a}), there is a unique isomorphism $\Phi^* \colon F \to \mathcal I_v^* (H)$ such that $\Phi^* \circ \varphi = \partial$. The restriction of $\Phi^*$ to $F_{\varphi}$ is the asserted isomorphism $\Phi$.
\end{proof}

For $i \in [1,2]$, let $H_i$ be a Krull monoid with divisor theory $\varphi_i \colon H_i \to F_i$. Then the monoid of principal ideals of $H_1$ and $H_2$ are isomorphic if and only if $H_1$ and $H_2$ have the same characteristic (roughly speaking, their class groups have to be isomorphic and the distribution of prime divisors in the classes has to be the same; for details see \cite[Theorem 2.5.4]{Ge-HK06a}). It is an open problem to characterize when $F_{\varphi_1}$ and $F_{\varphi_2}$ are isomorphic. We continue with an example.

\smallskip
\begin{example} \label{4.4}~

1. Let $D$ be a Dedekind domain. Then $\mathcal I_v^* (D) = \mathcal I^* (D)$ is the monoid of (usual) nonzero ring ideals with ideal multiplication and $\partial \colon D^{\bullet} \to \mathcal I^* (D) = F$, defined by $\partial (a) = aD$ for all $a \in D^{\bullet}$, is a divisor theory. For two nonzero ideals $I$ and $J$, we have $J \mid I$ if and only if $I \subset J$. Thus, for the conductor submonoid $F_{\partial}$ we have
\[
F_{\partial} = \{ I \in \mathcal I^* (D) \colon I \ \text{is contained in a proper principal ideal} \} \subset \mathcal I^* (D) \,.
\]

2. There are also non-commutative Krull monoids which give rise to conductor submonoids of factorial (commutative) monoids of ideals. Since non-commutative monoids are beyond the setting of the present article, we do not introduce  the required machinery but just give a brief sketch and associated references. Let $H$ be a normalizing Krull monoid. Then the monoid $\mathcal I_v^* (H)$ of $v$-invertible $v$-ideals is free abelian and we have a divisor theory $\partial \colon H \to \mathcal I_v^* (H)$ (\cite[Theorem 4.13]{Ge13a}). The theory  of non-commutative Dedekind and Krull rings is presented in \cite{Re03, Je-Ok07a, Mc-Ro01a}.
\end{example}

\smallskip
\begin{theorem} \label{4.5}
Let all notation be as in Definition \ref{4.2}.
The restriction $\boldsymbol \beta^* = \widetilde{\boldsymbol \beta}|_{F_{\varphi}} \colon F_{\varphi} \to \mathcal F_{\iota} (G_0)$ is a transfer homomorphism and for the catenary degree in the fibres we have $\mathsf c (F_{\varphi}, \boldsymbol \beta^*) \le 2$. 
\end{theorem}

\begin{proof}
Let all notation be as above. In particular, we have $F_{\varphi} \subset \mathcal F (P)$.
By construction, $\boldsymbol \beta^*$ is a monoid epimorphism from $F_{\varphi}$ to $\mathcal F_{\iota} (G_0)$, which satisfies {\bf (T1)}. In order to verify {\bf (T2)}, let $u \in F_{\varphi}$, $B, C \in \mathcal F_{\iota} (G_0)\setminus\{1\}$ with $\boldsymbol \beta^* (u)= BC$. Then we may write $u$ in the form $u = p_1 \cdot \ldots \cdot p_k q_1 \cdot \ldots \cdot q_{\ell}$, with $p_1, \ldots, p_k, q_1, \ldots, q_{\ell} \in P$ such that $B = [p_1] \cdot \ldots \cdot [p_k]$, and $C = [q_1] \cdot \ldots \cdot [q_{\ell}]$. Since $B, C \in \mathcal F_{\iota} (G_0)$, after renumbering if necessary, we may assume that  there are $k' \in [1,k]$ and $\ell' \in [1, \ell]$ such that $B' = [p_1] \cdot \ldots \cdot [p_{k'}]$ and $C' =  [q_1] \cdot \ldots \cdot [q_{\ell'}]$ are zero-sum sequences. Since $\boldsymbol \beta \colon H \to \mathcal B (G_0)$ is a transfer homomorphism, we infer that $v' = p_1 \cdot \ldots \cdot p_{k'} \in \varphi (H \setminus H^{\times})$ and $w' = q_1 \cdot \ldots \cdot q_{\ell'} \in \varphi (H \setminus H^{\times})$. Thus, $v = v'p_{k'+1} \cdot \ldots \cdot p_k \in F_{\varphi}$, $w = w'q_{\ell'+1} \cdot \ldots \cdot q_{\ell} \in F_{\varphi}$, $\boldsymbol \beta^* (v)=B$, $\boldsymbol \beta^* (w)=C$, and $u=vw$. Thus, {\bf (T2)} holds.

\smallskip
In order to verify that $\mathsf c (F_{\varphi}, {\boldsymbol \beta^*}) \le 2$, let $a = p_1 \cdot \ldots \cdot p_{\ell} \in F_{\varphi}$, where $\ell \in \N$ and $p_1, \ldots, p_{\ell} \in P$,  and  $z, z' \in \mathsf Z (a)$ with $\overline{{\boldsymbol \beta^*}} (z) = \overline{{\boldsymbol \beta^*}} (z')$. We have to show that there is a $2$-chain $z=z_0, z_1, \ldots, z_k=z' \in \mathsf Z (a)$ between $z$ and $z'$ with $\overline{{\boldsymbol \beta^*}} (z) = \overline{{\boldsymbol \beta^*}} (z_i)$ for all $i \in [1,k]$.
Let $z = u_1 \cdot \ldots \cdot u_m$ and $z' = u_1' \cdot \ldots \cdot u_n'$ with $u_1,\ldots,u_m,u'_1,\ldots,u'_n\in \mathcal A(F_{\varphi})$. Since $\overline{{\boldsymbol \beta^*}} (z) = \overline{{\boldsymbol \beta^*}} (z')$, it  follows that $m=n$ and, after renumbering if necessary, that ${{\boldsymbol \beta^*}} (u_i)={{\boldsymbol \beta^*}} (u'_i)$ for all $i\in [1,m]$. In particular, the lengths of these sequences are equal:  $s_i:=|{{\boldsymbol \beta^*}} (u_i)|=|{{\boldsymbol \beta^*}} (u'_i)|$  for all $i\in [1,m]$. Let $\ell=\sum_{i=1}^{m}|{{\boldsymbol \beta^*}}(u_i)|=\sum_{i=1}^{m}|{{\boldsymbol \beta^*}}(u'_i)|$, and let $I_j=[s_0+\ldots+s_{j-1}+1,s_0+\ldots+s_{j-1}+s_j]$ for $j\in [1,m]$, where $s_0:=0$, so that $[1,\ell]=I_1\sqcup I_2\sqcup \ldots\sqcup I_m$ is a disjoint union of $[1,\ell]$ into $m$ consecutive intervals with lengths $|I_j|=s_j$ for $i\in [1,m]$. Then
\begin{align*}&a=u_1\cdot\ldots\cdot u_m=u'_1\cdot\ldots\cdot u'_m=p_1\cdot\ldots\cdot p_\ell\quad \mbox{ with }\\
&u_j=\prod_{i\in I_j} p_i\;\und\; u_j'=\prod_{i\in I_j}p_{\alpha(i)}\quad \mbox{ for $j\in [1,m]$},\end{align*} where $p_1,\ldots,p_\ell\in P$ and $\alpha$ is some permutation of $[1,\ell]$.
Since $\prod_{i\in I_j} [p_i]={{\boldsymbol \beta^*}}(u_j)={{\boldsymbol \beta^*}}(u'_j)=\prod_{i\in I_j} [p_{\alpha (i)}]$ for all $j\in [1,m]$, we can w.l.o.g. assume $$[p_i]=[p_{\alpha(i)}]\quad\mbox{ for all $i\in[1,\ell]$}.$$
Moreover, we can assume the permutation $\alpha$ is chosen, subject to prior constraints, so as  to maximize the number of $i\in [1,\ell]$ with $\alpha(i)=i$.
If $\alpha(i)=i$ for all $i\in [1,\ell]$, then $u_j=u'_j$ for all $j \in [1,m]$, whence $z=z'$ with there nothing to show. Therefore there must be some $s\in [1,\ell]$ with $\alpha(s)\neq s$, say with $s\in I_{j_1}$ and $\alpha(s)\in I_{j_2}$ where $j_1,\,j_2\in [1,m]$.
Then $$[p_{\alpha(s)}]=[p_s]=[p_{\alpha^{-1}(s)}]$$ with $\alpha^{-1}(s)\neq s$.
If $j_1=j_2$, then we can modify the permutation $\alpha$ by sending $s\mapsto \alpha^2(s)$ and $\alpha(s)\mapsto \alpha(s)$ (rather than $s\mapsto \alpha(s)$ and $\alpha(s)\mapsto \alpha^2(s)$), with $i\mapsto \alpha(i)$ for all $i\notin \{s,\alpha(s)\}$, and this new permutation will contradict the maximality of $\alpha$. Therefore $j_1\neq j_2$.
 Take the factorization $z=u_1\cdot\ldots\cdot u_m$ and replace $u_{j_1}$ by  $v_{j_1}:=u_{j_1} p_s^{-1} p_{\alpha(s)}$ and replace $u_{j_2}$ by $v_{j_2}:=u_{j_2} p_{\alpha(s)}^{-1} p_s$, leaving all other atoms alone, $v_{j}:=u_j$ for all $j\notin \{j_1,j_2\}$.
 Since $[p_s]=[p_{\alpha(s)}]$, we have $v_{j_1}, v_{j_2}\in F_{\varphi}$.
 By construction,  ${{\boldsymbol \beta^*}}(v_j)={{\boldsymbol \beta^*}}(u_j)$ for all $j\in[1,m]$. Consequently, since $v_{j_1},v_{j_2}\in F_{\varphi}$,  it follows  that $v_{j_1},v_{j_2}\in \mathcal A(F_{\varphi})$.
 Then $z_1=v_1\cdot\ldots\cdot v_m$ is another factorization of $a$  with $\overline{{\boldsymbol \beta^*}} (z)=\overline{{\boldsymbol \beta^*}} (z_1)$.  Since $z_1$ was constructed by replacing two atoms in the factorization $z$ by two new atoms, we have $\mathsf d(z,z_1)\leq 2$. Moreover, defining the permutation $\alpha_1$ for $z_1$ analogous to how we defined $\alpha$ for $z$, we find that the number of $i\in[1,\ell]$ with $\alpha(i)=i$ has increased by at least one (since $\alpha_1(i)=i$ for any $i\in[1,\ell]$ with $\alpha(i)=i$, while also $\alpha_1(s)=s$). Thus, iterating this process, we construct the desired $2$-chain $z=z_0, z_1, \ldots, z_k=z' \in \mathsf Z (a)$ between $z$ and $z'$ after a finite number of steps.
\end{proof}

Let $G$ and $G_0$ be as in Definition \ref{4.2}. 
If $G$ is finite, then  $\mathcal F (G) \setminus \mathcal F_{\iota} (G)$ is finite. Thus, in the terminology of \cite{Ba23a}, $\mathcal F_{\iota} (G)$ is a complement-finite ideal of $\mathcal F (G)$. Our next result reveals the algebraic structure of $F_{\iota} (G)$ in case of finite groups.

\smallskip
\begin{proposition} \label{4.6}
Let $G$ be an  abelian group and let $G_0 \subset G$ be a nonempty subset. 
\begin{enumerate}
\item If $G_0$ is a finite set of torsion elements, then the following statements hold.
      \begin{enumerate}
      \item[(a)] $F_{\iota} (G_0)$ is finitely generated.

      \item[(b)] $s$-$\spec \big(F_{\iota} (G_0) \big)$ is finite.

      \item[(c)] $F_{\iota} (G_0)$ is a $G$-monoid.

      \item[(d)] $F_{\iota} (G_0)$ is a \C-monoid in $\mathcal F (G_0)$.
      \end{enumerate}

\smallskip
\item The following statements are equivalent.
      \begin{enumerate}
      \item[(a)] $F_{\iota} (G)$ is finitely generated.

      \item[(b)] $s$-$\spec \big(F_{\iota} (G) \big)$ is finite.

      \item[(c)] $F_{\iota} (G)$ is a $G$-monoid.

      \item[(d)] $F_{\iota} (G)$ is a \C-monoid in $\mathcal F (G)$.

      \item[(e)] $G$ is finite.
      \end{enumerate}
\end{enumerate}
\end{proposition}

\begin{proof}
1. By construction, $F_{\iota} (G_0) \subset \mathcal F (G_0)$ is a conductor submonoid. Suppose that $G_0$ is a finite, nonempty set of torsion elements. Then $\mathcal B (G_0)$ is finitely generated by \cite[Theorem 3.4.2]{Ge-HK06a}. Since every element of $G_0$ is torsion, we obtain that $\mathcal F(G_0)\setminus F_{\iota} (G_0)$ is finite.
In view of  the atoms of $F_{\iota} (G_0)$ are subsequences of the product set $\mathcal A (\mathcal B (G_0))  (\mathcal F(G_0)\setminus F_{\iota} (G_0))$, we have that $F_{\iota} (G_0)$ is finitely generated. Furthermore, (b), (c), and (d) hold true by Theorem \ref{3.5}.

2. It remains to show that for an infinite group, none of the properties (a), (b), (c) or (d) hold true. Suppose that $G$ is infinite.

(i) We assert  that the class semigroup $\mathcal C(\mathcal F_{\iota}(G), \mathcal F(G))$ is infinite, which implies that $F_{\iota} (G)$ is not a \C-monoid.
	Let $H=\mathcal F_{\iota}(G)$ and $F=\mathcal F(G)$.
	It suffices to show the set $\{[g]_H^F\colon g\in G\}\subset \mathcal C(H,F)$ is infinite.
	In fact, for any distinct $g,h\in G$, we have $-g\in (g^{-1}H\cap F)\setminus (h^{-1}H\cap F)$, whence $[g]_H^F\neq [h]_H^F$. 

(ii) We assert that $F_{\iota} (G)$ is not a \G-monoid. Since finitely generated monoids and all monoids $H$ with $s$-$\spec (H)$ finite are \G-monoids, it follows that $F_{\iota} (G)$ is not finitely generated and $s$-$\spec \big(F_{\iota} (G) \big)$ is infinite. Assume to the contrary that $F_{\iota} (G)$ is a \G-monoid. Then \cite[Lemma 2.7.7]{Ge-HK06a} implies that there is an element $A \in F_{\iota} (G)$ such that the divisor-closed submonoid generated by $A$ is equal to $F_{\iota} (G)$. Since the support of $A$ in $\mathcal F (G)$ is finite, this is a contradiction. 
\end{proof}

\smallskip
\begin{theorem} \label{4.7}
Let $H$ be a Krull monoid with divisor theory $\varphi \colon H \to F$,  class group $G$, and let $G_0 \subset G$ denote the set of classes containing prime divisors. 
\begin{enumerate}
\item If $G_0$ is a finite set of torsion elements, then $F_{\varphi}$ is a \C-monoid defined in $F$.

\item Suppose that $G_0=G$. Then $F_{\varphi}$ is a \C-monoid defined in $F$ if and only if $G$ is finite.
\end{enumerate}
\end{theorem}

\begin{proof}
We assert that two elements $y, y' \in F$ are $F_{\varphi}$-equivalent if and only if $\boldsymbol \beta^* (y)$ and $\boldsymbol \beta^* (y')$ are $F_{\iota} (G_0)$-equivalent. For this we have to verify that for all elements $z \in F$, we have
\[
z \in F_{\varphi} \quad \text{if and only if} \quad \boldsymbol \beta^* (z) \in F_{\iota} (G_0) \,,
\]
which holds true by construction.

1. If $G_0$ is a finite set of torsion elements, then, by Proposition \ref{4.6}.1, the number of $\big(F_{\iota} (G_0), \mathcal F (G_0) \big)$-congruence classes is finite, whence the number of $(F_{\varphi}, F)$-congruence classes is finite and so $F_{\varphi}$ is a C-monoid. 

2. This follows from Proposition \ref{4.6}.2.
\end{proof}

\smallskip
\section{Arithmetic of conductor submonoids} \label{5}
\smallskip

The first  result in this section (Theorem \ref{5.1}) was proved before in several special cases. As outlined in Example \ref{3.9}, the additive monoid $H = \N_{\ge \alpha}^s \cup \{\boldsymbol 0\} \subset (\N_0^s, +)$, with $s \ge 2$ and $\alpha \in \N$, is a conductor submonoid. It has catenary degree $\mathsf c (H)=3$ by \cite[Example 3.1.8]{Ge-HK06a}. Furthermore, seminormal finitely primary monoids are conductor submonoids and they have catenary degree equal to three by \cite[Theorem 3.8]{Ge-Ka-Re15a}. Finally, gap absorbing monoids have catenary degree at most $4$ and all their length sets are intervals by \cite[Theorems 4.4 and 5.1]{Ci-GS-LL25a}. All these results are generalized resp. refined by the following theorem.

\smallskip
\begin{theorem} \label{5.1}
Every conductor submonoid of a factorial monoid has catenary degree at most $3$. In particular, all length sets are intervals.
\end{theorem}

\begin{proof}
The in particular statement follows from the Inequality \eqref{dist-2}.
Let $F=F^{\times}\times \mathcal F(P)$ be a factorial monoid and let $H\subset F$ be a conductor submonoid. By Lemma \ref{3.4}.3, we may assume that $H$ is reduced, whence $H^{\times}=\{1\}$.
 Let $H_0=H\cap \mathcal F(P)$.
Then $H_0$ is a conductor submonoid of $\mathcal F(P)$ and $H=F^{\times}H_0\cup \{1\}$. It follows that $\mathcal A(H)=\{\epsilon u\colon u\in \mathcal A(H_0), \epsilon\in F^{\times}\}$. We proceed by the following claim.

\medskip
\noindent{\bf Claim A. }{\it We have $\mathsf c(H_0)\le 3$.}
\smallskip

\begin{proof}[Proof of Claim A]

	Assume to the contrary there exist $b\in H_0$ and two factorizations $ z$ and $ z'$ of $b$ with $| z|\ge | z'|$ such that there is no $3$-chain between $ z$ and $z'$. We may assume that $b$ is a counterexample with $|b|$ minimal.
	Then $| z|\ge 4$. It follows from Lemma \ref{3.4}.4, Lemma  \ref{minimal}, and Theorem \ref{3.8} that
		\begin{equation}\label{e3.4}
	\begin{aligned}
	&\text{for every  $a\in \mathcal A(H_0)\setminus \mathcal M(H_0)$, there exists $p\in P$ such that $p^{-1}a\in \mathcal A(H_0)$, }\\
\text{ and }\quad	&\text{for every  $b\in \mathcal A(H_0)$ and every $p\in P$, we have  $pb\in \mathcal A(H_0)\cup \mathcal A(H_0)\mathcal A(H_0)$. }
	\end{aligned}
	\end{equation}  
	Therefore there exists a factorization $ z_1=u_1\cdot\ldots\cdot u_{\ell}$ of $b$ with $\ell\ge 4$ and $u_1,\ldots,u_{\ell-1}\in \mathcal M(H_0)$, $u_{\ell}\in \mathcal A(H_0)$, such that there is a $3$-chain between $ z$ and $ z_1$. Let $ z_2=v_1\cdot\ldots\cdot v_k$ be a factorization of $b$ with $|\gcd(v_1,u_1)|$ maximal such that there is a $3$-chain between $ z'$ and $ z_2$. We may further assume that $z_2$ is such a factorization with the additional condition that $|v_1|$ is minimal. 
	
	If $k=2$, then it follows from $u_1\mid v_1v_2$ that we have decompositions 
	\[
	v_1=h_1h_2 \text{ and } v_2=f_1f_2 \text{ such that } h_1f_1=u_1 \text{ and }h_2f_2\in H_0\setminus \{1\}\,.
	\]
	Since $v_1,v_2\in \mathcal A(H_0)$, we obtain $h_2,f_2\in \mathcal G(H_0)\cup \mathcal A(H_0)$. By Lemma \ref{3.7} and Theorem \ref{3.8}, we have $h_2f_2\in \mathcal A(H_0)\cup \mathcal A(H_0)\mathcal A(H_0)$, whence there exists a factorization $z^*=u_1w_1\cdot\ldots\cdot w_t$ of $b$ with $t\le 2$. It follows by the minimality of $|b|$ that there exists a $3$-chain between $u_2\cdot\ldots\cdot u_{\ell}$  and $w_1\cdot\ldots\cdot w_t$, whence there are $3$-chains between $z_1$ and $z^*$, and hence between $z$ and $z'$, a contradiction. Thus $k\ge 3$.

	If $u_1\mid v_1$, then  \eqref{e3.4} implies that there exists a factorization $z^*=u_1w_1\cdot\ldots\cdot w_t$ of $b$ such that there is a $3$-chain between $z_2$ and $z^*$. By the minimality of $|b|$, there exists a $3$-chain between $u_2\cdot\ldots\cdot u_{\ell}$ and $w_1\cdot\ldots\cdot w_t$.
	It follows that there exist $3$-chains between $z_1$ and $z^*$, and hence between $z$ and $z'$, a contradiction.
	
	Suppose $u_1\nmid v_1$. Since $u_1\in \mathcal M(H_0)$, we have $\gcd(u_1,v_1)\in \mathcal G(H_0)=\mathcal F(P)\setminus H_0$ and  $v_1=v^*\gcd(u_1,v_1)$ for some $1\neq v^*\in \mathcal F(P)$. If there exists $p\in P$ with $p\mid v^*$ such that $p^{-1}v_1\in \mathcal A(H_0)$, then by moving $p$ from $v_1$ to $v_2$, it follows from \eqref{e3.4} that there is a new factorization $z^*=(p^{-1}v_1)w_1\cdot\ldots\cdot w_{t}$ with $|p^{-1}v_1|<|v_1|$ and $\gcd(p^{-1}v_1, u_1)=\gcd(v_1,u_1)$ such that there is a $3$-chain between $z^*$ and $z'$, a contradiction to the minimality of $|v_1|$. Thus 
	\begin{equation}\label{e3.6}
	\text{for every $p\in P$ with $p\mid v^*$, we have $p^{-1}v_1\in \mathcal G(H_0)$.}
	\end{equation}
	After renumbering if necessary, we may assume that there exists $p_0\in P$ with $p_0\mid v_2$ such that $p_0\gcd(u_1,v_1)\mid u_1$.
	Thus either $p_0^{-1}v_2\in \mathcal A(H_0)$ or $p_0^{-1}v_2\in \mathcal G(H_0)$. It follows from Theorem \ref{3.8} that  $\mathcal G(H_0)\mathcal A(H_0)\subset \mathcal A(H_0)\cup \mathcal A(H_0)\mathcal A(H_0)$, whence
	\begin{equation}\label{e3.7}
		p_0^{-1}v_2v\in \mathcal A(H_0)\cup \mathcal A(H_0)\mathcal A(H_0) \text{ for every }v\in \mathcal A(H_0)\,.
	\end{equation}
	In particular, we have $p_0^{-1}v_2v_3\in \mathcal A(H_0)\cup \mathcal A(H_0)\mathcal A(H_0)$.
	If $p_0v_1\in \mathcal A(H_0)$, then there exists a factorization $z^*=(p_0v_1)w_1\cdot\ldots\cdot w_t$ of $b$ with $|\gcd(p_0v_1,u_1)|>|\gcd(v_1,u_1)|$ such that there is a $3$-chain between $z^*$ and $z'$, a contradiction to the maximality of $|\gcd(v_1,u_1)|$. Thus $p_0v_1=p_0v^*\gcd(v_1,u_1)=x_1x_2$ with $x_1,x_2\in \mathcal A(H_0)$. It follows from \eqref{e3.6} that $x_1\mid p_0\gcd(v_1,u_1)\mid u_1$, whence $x_1=p_0\gcd(v_1,u_1)=u_1$ and $v^*=x_2$. 
	By \eqref{e3.7}, we have $p_0^{-1}v_2x_2\in \mathcal A(H_0)\cup \mathcal A(H_0)\mathcal A(H_0)$. Therefore there exists a factorization $z^*=u_1w_1\cdot\ldots\cdot w_t$ of $b$ with $|\gcd(u_1,u_1)|>|\gcd(v_1,u_1)|$ such that there is a $3$-chain between $z^*$ and $z'$, a contradiction to the maximality of $|\gcd(v_1,u_1)|$. \qedhere[Claim A]
\end{proof}

 Let $a\in H$ and let $z_1,z_2$ be two factorizations of $a$, say
\[
z_1=(\epsilon_1u_1)\cdot\ldots\cdot (\epsilon_ku_k) \text{ and }
z_2=(\beta_1v_1)\cdot\ldots\cdot (\beta_lv_l)\,,
\]
where $u_1,\ldots,u_k,v_1,\ldots,v_l\in \mathcal A(H_0)$ and $\epsilon_1,\ldots,\epsilon_k,\beta_1,\ldots,\beta_l\in F^{\times}$ with $\epsilon:=\epsilon_1\cdot\ldots\cdot\epsilon_k=\beta_1\cdot\ldots\cdot\beta_l$.
Let $z_1':=(\epsilon u_1)u_2\cdot\ldots\cdot u_k$ and $z_2':=(\epsilon v_1)v_2\cdot\ldots\cdot v_l$. Then there are $2$-chains between $z_1$ and $z_1'$ and between $z_2$ and $z_2'$.
By Claim A, there is a $3$-chain $y_0, y_1,\ldots,y_s$ between $y_0=u_1\cdot\ldots\cdot u_k$ and $y_s:=v_1\cdot\ldots\cdot v_l$. By inductively substituting a proper atom $w$ in $y_i$ by $\epsilon w$ for each $i\in [1,s]$, we obtain a $3$-chain between $z_1'$ and $z_2'$, which completes the proof.
\end{proof}

The result of the previous theorem is in contrast to  the arithmetic of Krull monoids (recall that, by Lemma \ref{3.4}, conductor submonoids that are distinct from the factorial monoid are not Krull). Let $H$ be a Krull monoid with divisor theory $\varphi \colon H \to F$, class group $G$, and suppose that each class contains a prime divisor. Then $F_{\varphi}$ is a conductor submonoid and its catenary degree is bounded by three. However, the catenary degree of $H$ is bounded by three if and only if $G$ is isomorphic to one of the following groups (\cite[Corollary 6.4.9]{Ge-HK06a})
\[
C_1, C_2, C_3, C_2 \oplus C_2, C_3 \oplus C_3 \,,
\]
where $C_n$ denotes a cylic group of order $n \in \N$. Moreover, if $G$ is infinite, then $\mathsf c (H)=\infty$ (compare with Theorem \ref{5.5}).

Theorem \ref{5.1} gives rise to two corollaries. The first one deals with the monoid of $v$-invertible $v$-ideals of  weakly Krull Mori monoids, whereas  the second one studies the monotone catenary degree of conductor submonoids. 

Let us discuss the assumptions of our first corollary. We  provide some properties of weakly Krull monoids and refer to \cite{HK98, Ge-HK06a, HK25a} for  background information.  A domain $D$ is a weakly Krull domain if and only if its multiplicative monoid $D^{\bullet}$ is a weakly Krull monoid. Noetherian Cohen-Macaulay domains are weakly Krull, whence all one-dimensional noetherian domains are weakly Krull. Suppose that $D$ is one-dimensional and noetherian. Then  the monoid $\mathcal I_v^* (D)$ of $v$-invertible $v$-ideals coincides with the monoid of invertible ideals (with usual ideal multiplication). If  $\widehat D$ is a finitely generated $D$-module, equivalently, $(\widehat D \DP D) \ne \{0\}$ (all this holds true for orders in algebraic number fields), then all localizations are finitely primary monoids. Note that, without the assumption that the localizations are conductor submonoids,  the conclusion (that the catenary degree is bounded by three) is not even true in general for orders in quadratic number fields (\cite{Re23a}). An analogue result, as given in Corollary \ref{5.2}, holds true for the monoid of $v$-invertible $v$-ideals of $v$-Marot weakly Krull Mori rings (these are commutative rings with zero-divisors; see \cite[Theorem 5.10]{Ba-Po25a}).

\smallskip
\begin{corollary} \label{5.2}
Let $H$ be a weakly Krull Mori monoid. If the localizations $H_{\mathfrak p}$ are conductor submonoids for all $\mathfrak p \in \mathfrak X (H)$, then the monoid of $v$-invertible $v$-ideals $\mathcal I_v^* (H)$ is a weakly factorial Mori monoid which is locally tame and has catenary degree at most three.
\end{corollary}

\begin{proof}
By \cite[Proposition 5.3]{Ge-Ka-Re15a}, there is an isomorphism
\[
\delta_H \colon \mathcal I_v^* (H) \to \coprod_{\mathfrak p \in \mathfrak X (H)} (H_{\mathfrak p})_{\red}
\]
and $\mathcal I_v^* (H)$ is a weakly factorial Mori monoid. If $\mathfrak p \in \mathfrak X (H)$, then $H_{\mathfrak p}$ is a conductor submonoid by assumption and then the same is true for the associated reduced monoid $(H_{\mathfrak p})_{\red}$ by Lemma \ref{3.4}. Therefore, all $(H_{\mathfrak p})_{\red}$ have catenary degree at most three. Since they are primary Mori monoids, they are C-monoids by Theorem \ref{3.5} and all C-monoids are locally tame by \cite[Theorem 3.3.4]{Ge-HK06a}.  Both properties, local tameness and the upper bound for the catenary degree, go over to the coproduct by \cite[Proposition 1.6.8]{Ge-HK06a}.
\end{proof}

The monotone catenary degree of a monoid is studied in two steps. We introduce the necessary concepts. 
Let $H$ be a BF-monoid. For subsets $X, Y \subset \mathsf Z (H)$, we set
\[
\mathsf d (X, Y) = \min \{ \mathsf d (x,y) \colon x \in X, y \in Y\} \,.
\] 
Let $a \in H$. For $k \in \N_0$, we denote by $\mathsf Z_k (a) \subset \mathsf Z (a)$ the set of factorizations of $a$ having length $k$.
If $k, \ell \in \mathsf L (a)$ are distinct, then $k$ and $\ell$ are called {\it adjacent} if there is no $m \in \mathsf L (a)$ with $k < m < \ell$ or with $\ell < m < k$. We set
\[
\mathsf c_{\adj} (a) = \sup \{ \mathsf d \big( \mathsf Z_k (a), \mathsf Z_{\ell} (a)\big) \colon k, \ell \in \mathsf L (a) \ \text{are adjacent} \}
\]
and the {\it adjacent catenary degree} of $H$ is defined as
\[
\mathsf c_{\adj} (H) = \sup \{\mathsf c_{\adj} (b) \colon b \in H\} \in \N_0 \cup \{\infty\} \,.
\]
Thus, $\mathsf c_{\adj} (H) = 0$ if and only if $H$ is half-factorial. By \cite[Theorem 4.4]{Ba23a}, a conductor submonoid $H \subset F=\mathcal F (P)$ with $F \setminus H$ finite is half-factorial if and only if $H=F$.
The {\it equal catenary degree} $\mathsf c_{\eq} (a)$ of $a$ is the smallest $N \in \N_{0} \cup \{\infty\}$ such that any two factorizations $z, z' \in \mathsf Z (a)$ with $|z|=|z'|$ can be concatenated by a monotone $N$-chain of factorizations of $a$. Then
\[
\mathsf c_{\eq} (H) = \sup \{\mathsf c_{\eq} (b) \colon b \in H\} \in \N_0 \cup \{\infty\}
\]
denotes the {\it equal catenary degree} of $H$. Thus,  $\mathsf c_{\eq} (H) = 0$ if and only if $H$ is length-factorial and length-factorial domains are factorial \cite{C-C-G-S21}. Furthermore, 
we have
\[
\mathsf c (a) \le \mathsf c_{\mon} (a) = \sup \{ \mathsf c_{\eq} (a), \mathsf c_{\adj} (a) \} \le \max \mathsf L (a) 
\]
and hence
\[
\mathsf c (H) \le \mathsf c_{\mon} (H) = \sup \{\mathsf c_{\eq} (H), \mathsf c_{\adj} (H) \} \,.
\]
The monotone catenary degree is finite for finitely generated monoids \cite{Fo06a},  for the monoid $\mathcal I_v^* (D)$ of a seminormal weakly Krull domain with nonzero conductor \cite[Theorem 5.8]{Ge-Ka-Re15a}, and others \cite{Ge-Gr-Sc-Sc10, Ge-Re19d}. But it fails to be finite in general for finitely primary monoids and for C-monoids (\cite{Fo-Ge05, Fo-Ha06b, Fo-Ha06a}).

\smallskip
\begin{corollary}~ \label{5.3}

\begin{enumerate}
\item If $H$ is a conductor submonoid of a factorial monoid, then $\mathsf c_{\adj} (H) \in \{0,3\}$.

\item For every $m \in \N$, there is a conductor submonoid $H_m$ of a finitely generated free abelian monoid $F_m$ with   $\mathsf c_{\eq} (H_m) \ge m$.
\end{enumerate}
\end{corollary}

\begin{proof}
1. As observed above, $\mathsf c_{\adj} (H) = 0$ if and only if $H$ is half-factorial. Suppose that $H$ is not half-factorial. It remains to show that $\mathsf c_{\adj} (a) = 3$ for all $a \in H$ with $|\mathsf L (a)| > 1$. Let $a \in H$ and let $k, \ell \in \mathsf Z (a)$ be adjacent lengths with $k < \ell$. We choose two factorizations $z, z' \in \mathsf Z (a)$ with $|z|=k$ and $|z'|=\ell$. By Theorem \ref{5.1}, there is a $3$-chain of factorizations concatenating $z$ and $z'$. Then the basic Inequality \eqref{dist-1} shows that $\mathsf d \big( \mathsf Z_k (a), \mathsf Z_{\ell} (a)\big) = 3$, whence $\mathsf c_{\adj} (a) = 3$. 

2. Let $m\ge 3$ and 
let $F_m=\mathcal F(P_m)$ be a free abelian monoid whose finite basis $P_m$ has at least $2m$ elements, say  $P \supset \{p_1,\ldots, p_{2m}\}$. We set $p_{2m+1}=p_1$ and let $H\subset F_m$ be the conductor submonoid generated by $\{p_ip_{i+1}\colon i\in [1,2m]
\}$. Now consider the element $a=p_1p_2\cdot\ldots\cdot p_{2m}\in H$.
It is easy to see that $\{p_ip_{i+1}\colon i\in [1,2m]
\}$ is  the set of all atoms of $H$ that have length $2$ in $F_m$. It follows that $a$ has only two factorization of length $m$, that is,
\[
a=(p_1p_2)(p_3p_4)\cdot\ldots\cdot (p_{2m-1}p_{2m})=(p_{2m}p_1)(p_2p_3)\cdot\ldots\cdot (p_{2m-2}p_{2m-1})\,,
\]
whence $\mathsf c_{\eq} (H) \ge  \mathsf c_{\eq}(a) \ge m$.
\end{proof}

We formulate the following conjecture.

\smallskip
\begin{conjecture} \label{5.4}
Let $F$ be a factorial monoid and let $H \subset F$ be a conductor submonoid. If $F_{\red}$ is finitely generated, then $\mathsf c_{\mon} (H) < \infty$.
\end{conjecture}

\smallskip
We end this paper with a detailed discussion of the arithmetic of conductor submonoids $F_{\varphi}$, associated to a Krull monoid with divisor theory $\varphi \colon H \to F$. We recall the definition of unions of length sets.
Let $H$ be a BF-monoid but not a group. For $k \in \N$, let 
\[
\mathcal U_k (H) = \bigcup_{k \in L, L \in \mathcal L (H)} L \quad \subset \quad \N_0
\]
denote the {\it union of length sets} containing $k$, and let $\rho_k (H) = \sup \mathcal U_k (H)$ be the $k$th {\it elasticity} of $H$.  Clearly, $\mathcal U_k (H)$ is the set of all $\ell \in \N$ for which there is an equation of the form
\[
u_1 \cdot \ldots \cdot u_k = v_1 \cdot \ldots \cdot v_{\ell} \,, \text{where} \ u_1, \ldots, u_k, v_1, \ldots, v_{\ell} \in \mathcal A (H) \,.
\]
Unions of length sets received a lot of attention in recent literature (e.g., \cite{Tr19a,Ge-Lo-Ki26}).   
The $k$th elasticities are finite for all finitely generated monoids, and Baeth studied their precise values  of $\rho_k ( F_{\iota} (G))$ for finite groups $G$ \cite[Theorem 5.5]{Ba23a}. But they can also be infinite for conductor submonoids $H \subset F$, where $F_{\red}$ is finitely generated. Indeed, 
if  $H = (p_1^{\alpha_1} \cdot \ldots \cdot p_s^{\alpha_s})F \cup H^{\times} \subset F$, with all notation as in Example \ref{3.9}.2, then it is easy to check that, for all $k \ge 2$,  $\rho_k (H) < \infty$ in case $s=1$, whereas $\rho_k (H) = \infty$ in case $s \ge 2$. 

If the class group of the Krull monoid $H$ is infinite, then we explicitly write down the full system $\mathcal L (F_{\varphi})$ of length sets. This has been done so far in very special cases of monoids and domains only (we refer to \cite[Chapter 4]{Ge-HK06a} for background on the system of length sets, and to \cite{Go19a, Fa-Wi24a} for recent progress).

\smallskip
\begin{theorem} \label{5.5}
Let $H$ be a Krull monoid with divisor theory $\varphi \colon H \to F$, non-trivial class group $G$, and let $G_0 \subset G$ denote the set of classes containing prime divisors. Then $F_{\varphi}$ is a non-half-factorial \FF-monoid without any prime elements and with $\mathsf c (F_{\varphi}) = \mathsf c_{\adj} (F_{\varphi}) = 3$.
\begin{enumerate}
\item If $G$ is  finite, then $F_{\varphi}$ is locally tame, $\mathsf c_{\mon} (F_{\varphi}) = \mathsf c_{\mon} (F_{\iota} (G_0)) < \infty$, and  $\mathcal U_k (F_{\varphi})$ is a finite interval for every $k \ge 2$.

\item If $G_0=G$ is infinite, then
\[
\mathcal L ( F_{\varphi}) = \big\{  \{k\} \colon k \in \N_0\big\} \cup \big\{ L \subset \N_{\ge 2} \colon L \ \text{is a finite interval} \big\} \,,
\]
\[
\mathcal U_k (F_{\varphi}) = \N_{\ge 2} \ \ \text{for all $k \ge 2$}, 
\quad \text{and} \quad \mathsf c_{\mon} (F_{\varphi}) = \mathsf c_{\eq} (F_{\varphi}) = \infty \,.
\]
\end{enumerate}
\end{theorem}

\begin{proof}
Since $F$ is reduced, $F_{\varphi}$ is an \FF-monoid by Lemma \ref{3.4}. By Definition \ref{4.2}, $F_{\varphi} \subset F$ is a conductor submonoid. 
By Theorem \ref{4.5}, we have  $\mathcal L (F_{\varphi}) = \mathcal L (F_{\iota} (G_0)) $. Since $G_0$ generates $G$ as a monoid and $G$ is nontrivial, $G_0$ does not consist of the zero element only. 

To show that $F_{\varphi}$ has no prime elements, let $u \in F_{\varphi}$ be an atom.  Since $\mathcal B (G_0) \ne \mathcal B (\{0\})$, there is an atom $v=q_1 \cdot \ldots \cdot q_n \in F_{\varphi}$ with $n \ge 2$ and primes $q_1, \ldots, q_n \in F$. Then $u$ divides $u^2v$, but $u$ does not divide $uq_1$ nor $uq_2 \cdot \ldots \cdot q_n$, whence $u$ is not prime.

In order to show that $F_{\iota} (G_0)$ is not half-factorial, we distinguish two cases. If $0 \in G_0$, we choose any atom $V = g_1 \cdot \ldots \cdot g_{\ell} \in \mathcal B (G_0)$ with $\ell \ge 2$. Then $S = 0^2V$ has a factorization of length three and, because $S = (0g_1)(0g_2 \cdot \ldots \cdot g_{\ell})$, also a factorization of length two. Suppose that $0 \notin G_0$. We choose two atoms $U, V \in \mathcal B (G_0)$ such that
\[
|UV| = \min \{ |W_1W_2| \colon W_1, W_2 \in \mathcal A ( \mathcal B (G_0) ) \} \ge 4 \,.
\]
We set $V = g_1 \cdot \ldots \cdot g_{\ell}$ with $\ell \ge 2$. Then $U^2V = \big( U g_1 \big) \big( U g_2 \cdot \ldots \cdot g_{\ell} \big)$. By the choice of $U$ and $V$, $Ug_1$ and $U g_2 \cdot \ldots \cdot g_{\ell}$ are atoms in $F_{\iota} (G)$, whence $U^2V$ has a factorization of length two.

Since   $F_{\varphi}$ is not half-factorial, Theorem  \ref{5.1} and Corollary \ref{5.3} imply that  $\mathsf c (F_{\varphi}) = \mathsf c_{\adj} (F_{\varphi}) = 3$.

\bigskip
1. Suppose that $G$ is finite. 

By Theorem \ref{4.7}, $F_{\varphi}$ is a \C-monoid defined in $F$ and \C-monoids are locally tame by \cite[Theorem 3.3.4]{Ge-HK06a}.
Since  $F_{\iota} (G_0)$ is finitely generated by Proposition \ref{4.6}.1,  $\mathsf c_{\mon} (F_{\iota} (G_0)) < \infty$ by \cite[Theorem 3.9]{Fo06a} (or by \cite[Theorem 3.1]{Ge-Re19d}) and all $\mathcal U_k (F_{\iota} (G_0)$ are finite by \cite[Theorem 3.1.4]{Ge-HK06a}. Since $\mathsf c_{\mon} (F_{\varphi}) \ge \mathsf c (F_{\varphi})=3$,   Lemma \ref{4.1} and Theorem \ref{4.5} imply that $\mathsf c_{\mon} (F_{\varphi}) = \mathsf c_{\mon} (F_{\iota} (G_0))$ . Since all length sets are intervals, the sets $\mathcal U_k (F_{\iota} (G_0)$ are intervals.

\bigskip
2. Suppose that $G_0=G$ is infinite.

\medskip
\noindent
{\bf Part 1: On  $\mathcal L (F_{\iota} (G))$.}  By convention, we have $\mathsf L (1) = \{0\}$ and obviously we have $\mathsf L (0^k) = \{k\}$ for all $k \in \N$. 
Since $H$ is a BF-monoid with $\mathsf c (F_{\varphi}) = 3$, all the other length sets are finite intervals of $\N_{\ge 2}$. Thus, we need to show that,  for any $k,\ell\in \N$ with $\ell> k\ge 2$, $[k,\ell]$ is a length set. Since all length sets are intervals, we only need to find some $S\in \mathcal F_{\iota}(G)$ such that $\min \mathsf L(S)=k$ and $\max \mathsf L(S)=\ell$. We distinguish three cases.

\smallskip
\noindent
CASE 1: $k=2$.

Since the Davenport constant $\mathsf D (G)$ is infinite, there is an atom $U \in \mathcal A ( \mathcal B (G) )$ of length $|U|=\ell$. Then $U$ is also an atom of $F_{\iota} (G)$, which implies that $\min \mathsf L \big( (-U)U \big) = 2$ and $\max \mathsf L \big( (-U)U \big) = |U|=\ell$.

\smallskip
\noindent
CASE 2: $\ell=k+1$.

We choose an element $g\in G\setminus\{0\}$ and consider the sequence $S=0^{k}(-g)g$. It is easy to see that $\max \mathsf L(S)=k+1$. Since for every atom $A$ of $\mathcal F_{\iota}(G)$, we have $\mathsf v_0(A)\le 1$, it follows from $S= 0^{k-2} \big(0g \big) \big(0(-g)\big)$ that $\min\mathsf L(S)=k$.

\smallskip

\noindent
CASE 3: $k \ge 3$ and $\ell > k+1$.

We set $n=\ell-k-1$ and distinguish two subcases.

\smallskip
\noindent
CASE 3.1: There is some  $g \in G$ with $\ord(g)>2n$. 

The sequence $S=0^kg^n(-ng)(-g)^n(ng)$ satisfies $\max \mathsf L(S)=k+n+1=\ell$. Since $\min \mathsf L(S)\ge k$ and 
	\[
	S= \big(0g^n(ng)\big) \, \big(0(-g)^n(ng) \big) \, 0^{k-2}\,,
	\]
	it follows that $\min\mathsf L(S)=k$.
	
\smallskip
\noindent
CASE 3.2:	Every element of $G$ has order at most $2n$.

Then $G$ has infinite rank. First, suppose that  $G$ is an elementary $2$-group and $[k,\ell]=[3,5]$. In this case we choose two independent elements $e_1,e_2$ and consider $S=0^2e_1^2e_2^2(e_1+e_2)^2$. Obviously, we have  $\max\mathsf L(S)=5$ and since $S=(0e_1e_2)(0e_1e_2)(e_1+e_2)^2$, it remains to verify that $\min \mathsf L(S)\neq 2$. Assume to the contrary that $S=(0T_1)(0T_2)$, where $T_1,T_2$ are zero-sum free subsequence such that $T_1T_2=e_1^2e_2^2(e_1+e_2)^2$. Since $T_1$ and $T_2$ must be squarefree, it follows that  $T_1=T_2=e_1e_2(e_1+e_2)$ have sum  zero, a contradiction.

\smallskip
Second, suppose  that either $G$ is not an elementary $2$-group, or $k\ge 4$, or $n\ge 2$.
We choose independent elements $e_1,\ldots, e_n$ of $G$ with $\ord(e_1)=\exp(G)$. We set $e_0=\sum_{i=1}^n(-e_i)$ and $A=\prod_{i=0}^ne_i$. The sequence $S=0^kA(-A)$ satisfies $\max \mathsf L(S)=k+|A|=\ell$ and it remains to verify that $\min\mathsf L(S)=k$. 
	 
If $e_0=e_n$, then $n=1$ and $\exp(G)=\ord(e_1)=2$, which, by our assumption, implies that $k\ge 4$. Since $S=0^ke_1^4$, $\min \mathsf L(S)\ge k$, and $S=(0e_1)^4 0^{k-4}$, it follows that
	 that $\min\mathsf L(S)=k$.
	 
If $e_0\neq e_n$, then $0e_0(-e_n)$ is an atom of $\mathcal F_{\iota}(G)$. Since  $\min \mathsf L(S)\ge k$ and 
\[
S = \big(0e_1\cdot\ldots\cdot e_n \big) \, \big(0(-e_0)\cdot\ldots\cdot (-e_{n-1}) \big) \,  \big(0e_0(-e_n) \big) \, 0^{k-3} \,,
\]
we infer that $\min\mathsf L(S)=k$. 

\medskip
\noindent
{\bf Part 2: On  $\mathcal U_k (F_{\varphi})$.} This claim easily follows from the structure of $\mathcal L (F_{\varphi})$. But, there is also a simple independent argument which we would like to provide.
Let $k \ge 2$. Since every finite nonempty subset of $\N_{\ge 2}$ is a length set of $H$ by \cite[Theorem 7.4.1]{Ge-HK06a}, it follows that $\mathcal U_k (H) = \N_{\ge 2}$. This means that for every $\ell \ge 2$ there are $u_1, \ldots, u_k,v_1, \ldots, v_{\ell} \in \mathcal A (H)$ such that 
\[
u_1 \cdot \ldots \cdot u_k = v_1 \cdot \ldots \cdot v_{\ell} \,.
\]
Since every atom of $H$ is an atom of $F_{\varphi}$, the above equation implies that $\mathcal U_k (F_{\varphi}) = \N_{\ge 2}$.

\medskip
\noindent
{\bf Part 3: On  the equal and the  monotone catenary degree.}
Since $\mathsf c_{\mon} (F_{\varphi}) \ge  \mathsf c_{\eq} (F_{\varphi})$, it remains to show that $\mathsf c_{\eq} (F_{\varphi}) = \infty$ and by Theorem \ref{4.5} it suffices to verify that $\mathsf c_{\eq} \big( \mathcal F_{\iota} (G) \big) = \infty$. 
Let $n\ge 2$. We assert that there exists some $a\in \mathcal F_{\iota}(G)$ such that $a$	has precisely two factorizations $z, z'$ of length  $|z|=|z'|=n+1$ and with $\mathsf d (z,z')=n+1$. This implies that  $\mathsf c_{\mathrm{eq}}(a)\ge n+1$. We distinguish two cases.

\smallskip
\noindent
CASE 1: There is an element $g \in G$ with $\ord (g) > n (2n^2)^n$.

We choose a prime $p \in [ n^2, 2n^2]$ and we define
\[
\begin{aligned}
 A =(-\sum_{i=0}^{n-1}p^ig)\prod_{i=0}^{n-1}p^ig \ , & \qquad A_0 = (\sum_{i=0}^{n-1}np^ig)\prod_{i=0}^{n-1}(-np^ig) \\
B_n=(\sum_{i=0}^{n-1}np^ig)(-\sum_{i=0}^{n-1}p^ig)^n \,, \quad & \text{and} \quad  B_i = (-np^ig)(p^ig)^{n} \ \text{ for each $i\in [0,n-1]$} \,.
\end{aligned} 
\]
We set $a=A_0A^n$ and $G_0 = \supp(A_0A)$, and we observe that the above defined atoms  are all atoms in $\mathcal B (G_0)$ of length $n+1$. 
Since $n+1$ is the minimal length of atoms in  $\mathcal B (G_0)$, it follows that 
\[
z=A_0A^n \quad \text{ and} \quad   z'=B_0B_1\cdot\ldots\cdot B_n
\]
are the only two factorizations of $a$ of length $n+1$.

\smallskip
\noindent
CASE 2: There is an $N \in \N$ such that $Ng = 0$ for every $g \in G$.

Then $G$ is a direct sum of cyclic groups and there is a prime $p$ such that the $p$-rank $\mathsf r_p (G)=\infty$. We choose independent elements $\{e_{i,j}\colon i, j\in [1,n]\}$ of order $p$ and  set 
\[
e_0=\sum_{i=1}^n\sum_{j=1}^ne_{i,j}, \quad -f_i=\sum_{j=1}^ne_{i,j}, \quad \text{and} \quad  -g_j=\sum_{i=1}^ne_{i,j} \quad \text{ for all $i,j \in [1,n]$} \,.
\] 
Thus, for all $i, j \in [1,n]$, 
\[
U_0:=e_0\prod_{i=1}^nf_i, \ \ V_0:=e_0\prod_{j=1}^ng_j, \ \  U_i:=f_i\prod_{j=1}^ne_{i,j}, \ \text{ and  } \ V_j:=g_j\prod_{i=1}^ne_{i,j}
\]
are  atoms of length $n+1$. We define $a=U_0U_1\cdot\ldots\cdot U_n$ and, by construction, we obtain that 
\[
z=U_0U_1\cdot\ldots\cdot U_n \quad \text{ and } \quad z'=V_0V_1\cdot\ldots\cdot V_n
\]
are the only two factorizations of $a$ of length $n+1$.
\end{proof}

\medskip
{\bf Acknowledgement.} We would like to thank O.~Esentepe for fruitful discussions and C.~Finocchiaro for providing Example \ref{3.11}.



\providecommand{\bysame}{\leavevmode\hbox to3em{\hrulefill}\thinspace}
\providecommand{\MR}{\relax\ifhmode\unskip\space\fi MR }
\providecommand{\MRhref}[2]{%
  \href{http://www.ams.org/mathscinet-getitem?mr=#1}{#2}
}
\providecommand{\href}[2]{#2}

\end{document}